\documentclass[10pt]{article}
\usepackage{amsmath}
\usepackage{amssymb}
\usepackage{latexsym,bm}
\usepackage{amsthm}
\usepackage{graphicx}

\usepackage[top=2cm,bottom=2cm,left=2.4cm,right=2cm]{geometry}
\usepackage{algorithm}
\usepackage{algorithmic}
\usepackage{cite}
\usepackage{caption}
\usepackage{subfigure}
\usepackage{bbding}
\allowdisplaybreaks
\usepackage{multirow}
\usepackage{threeparttable}
\usepackage{slashbox}

\newtheorem{theorem}{Theorem}[section]

\newtheorem{lemma}{Lemma}[section]

\theoremstyle{definition}
\newtheorem{definition}{Definition}[section]
\newtheorem{remark}{Remark}[section]
\numberwithin{equation}{section}
\theoremstyle{example}

\numberwithin{equation}{section}

\usepackage[english]{babel}
 \UseRawInputEncoding

\begin{document}

\title{Four-operator splitting algorithms for solving monotone inclusions}

\author{Jinjian Chen$^{a}$, Yuchao Tang$^{a}$
\\
{\small ${^a}$ Department of Mathematics, Nanchang University,  Nanchang 330031, P.R. China }}

 \date{}

\maketitle

{}\textbf{Abstract.}\ Monotone inclusions involving the sum of three maximally monotone operators or more have received much attention in recent years.
In this paper, we propose three splitting algorithms for finding a zero of the sum of four monotone operators, which are two maximally monotone operators, one monotone Lipschitz operator, and one cocoercive operator. These three splitting algorithms are based on the forward-reflected-Douglas-Rachford splitting algorithm, backward-forward-reflected-backward splitting algorithm, and backward-reflected-forward-backward splitting algorithm, respectively. As applications, we apply the proposed algorithms to solve the monotone inclusions problem involving a finite sum of maximally monotone operators. Numerical results on the Projection on Minkowski sums of convex sets demonstrate the effectiveness of the proposed algorithms.



\section{Introduction}

Monotone inclusions play an important role in studying various convex minimization problems, which arise in signal and image processing, medical image reconstruction and machine learning, etc. The traditional operator splitting algorithms are used to solve monotone inclusions involving the sum of two maximally monotone operators, where one of them is cocoercive or monotone Lipschitz. The most popular operator splitting algorithms include the Douglas-Rachford splitting algorithm \cite{lionsandmercier1979,Eckstein1992}, the forward-backward splitting algorithm \cite{lionsandmercier1979}, and the forward-backward-forward splitting algorithm \cite{Tseng2000SIAM}. These operator splitting algorithms and their variants have been extensively studied. See, e.g., \cite{Combettes2004Optimization,combettes2007,Combettes2014Optimization,Bonettini2016SIAMJSC,Eckstein2017MP,Attouch2018,CuiJIA2019,Attouch2019AMO} and references therein. To deal with monotone inclusions of the sum of three maximally monotone operators or more, several new operator splitting algorithms have been proposed. See, for example \cite{Raguet-SIAM-2013,Banert2012,davis2015,Raduet2015SIAMJIS,briceno2015Optim,Briceno2015JOTA,Arias2017A,Malitsky2020SIAMJO,Cevher2020SVVA,Ryu2020JOTA,Rieger2020AMC,Yu-tang-2020,Chenjj2021}.

Let $\mathcal{H}$ be a real Hilbert space and $m\geq 1$ be an integer. For each $i\in \{1,\cdots,m\}$, let $A_{i}:\mathcal{H}\rightarrow 2^\mathcal{H}$ be maximally monotone. Let $A:\mathcal{H}\rightarrow 2^\mathcal{H}$ be maximally monotone, $B:\mathcal{H}\rightarrow \mathcal{H}$ be monotone and $L$-Lipschitz with $L>0$, and $C:\mathcal{H}\rightarrow \mathcal{H}$ be $\beta$-cocoercive, for some $\beta>0$. Raguet et al. \cite{Raguet-SIAM-2013} first proposed a generalized forward-backward splitting algorithm for solving the following monotone inclusions problem:
\begin{equation}\label{problem1-m-maximally}
\textrm{find}\quad x\in \mathcal{H}\quad\textrm{such that}\quad 0\in \sum_{i=1}^{m}A_{i}x+Cx.
\end{equation}
In particular, the generalized forward-backward splitting algorithm reduces to the classical forward-backward splitting algorithm when $m=1$. When the cocoercive operator $C$ is replaced by a monotone Lipschitz operator $B$, Banert \cite{Banert2012} considered the following monotone inclusions problem£º
\begin{equation}\label{problem2-m-maximally}
\textrm{find}\quad x\in \mathcal{H}\quad\textrm{such that}\quad 0\in \sum_{i=1}^{m}A_{i}x+Bx.
\end{equation}
As a consequence, Banert \cite{Banert2012} proposed a relaxed forward-backward splitting algorithm. On the other hand, Davis and Yin \cite{davis2015} introduced a so-called three-operator splitting algorithm for solving the following monotone inclusions
\begin{equation}\label{problem1-three-operator}
\textrm{find}\quad x\in \mathcal{H}\quad\textrm{such that}\quad 0\in A_{1}x+A_{2}x+Cx.
\end{equation}
Brice\~no-Arias \cite{briceno2015Optim} studied the three-operator monotone inclusions (\ref{problem1-three-operator}), where one of the maximally monotone operators is the normal cone of a closed vector subspace. It is worth mentioning that the generalized forward-backward splitting algorithm \cite{Raguet-SIAM-2013} could be recovered by both of the operator splitting algorithms proposed by \cite{davis2015} and \cite{briceno2015Optim}.

In recent years, many authors studied the following three-operator monotone inclusions problem.
\begin{equation}\label{problem2-three-operator}
\textrm{find}\quad x\in H\quad\textrm{such that}\quad 0\in Ax+Bx+Cx.
\end{equation}
In particular, Brice\~no-Arias and Davis \cite{Arias2017A} first proposed a so-called forward-backward-half-forward splitting algorithm, which combined the classical forward-backward splitting algorithm  and Tseng's forward-backward-forward splitting algorithm. Recently, a semi-forward-reflected-backward splitting algorithm has been proposed by Malitsky and Tam  \cite{Malitsky2020SIAMJO}, and a semi-reflected-forward-backward splitting algorithm has been proposed by Cevher and V\~u \cite{Cevher-2019-arxiv,Cevher2020SVVA}. In addition, Yu et al. \cite{Yu-tang-2020} introduced an outer reflected forward-backward splitting algorithm to solve this problem as well.

If the cocoercive operator of (\ref{problem1-three-operator}) is relaxed to monotone Lipschitz operator, then the problem  (\ref{problem1-three-operator}) is reformulated as the following monotone inclusions:
\begin{equation}\label{problem3-three-operator}
\textrm{find}\quad x\in H\quad\textrm{such that}\quad 0\in A_{1}x+A_{2}x+Bx.
\end{equation}
Ryu and V\~u \cite{Ryu2020JOTA} proposed a so-called forward-reflected-Douglas-Rachford splitting algorithm, which is defined by,
\begin{equation}\label{algorithm1-three-operator}
\left\{
\begin{aligned}
& x_{n+1}=J_{\gamma A_{2}}(x_{n}-\gamma u_{n}-\gamma(2Bx_{n}-Bx_{n-1}))\\
& y_{n+1}=J_{\lambda A_{1}}(2x_{n+1}-x_{n}+\lambda u_{n})\\
& u_{n+1}=u_{n}+\frac{1}{\lambda}(2x_{n+1}-x_{n}-y_{n+1}),
\end{aligned}
\right.
\end{equation}
where $\lambda >0$ and $\gamma\in (0,\frac{\beta}{1+2\beta L})$. Besides, Rieger and Tam \cite{Rieger2020AMC} proposed two splitting algorithms to solve the problem (\ref{problem3-three-operator}). One is the Backward-forward-reflected-backward splitting algorithm, which combines the Forward-reflected-backward splitting algorithm and the Douglas-Rachford splitting algorithm. The iterative scheme is given by
\begin{equation}\label{algorithm2-three-operator}
\left\{
\begin{aligned}
& x_{n+1}=J_{\gamma A_{1}}z_{n}\\
& y_{n+1}=J_{\gamma A_{2}}(2x_{n+1}-z_{n}-2\gamma By_{n}+\gamma By_{n-1})\\
& z_{n+1}=z_{n}+y_{n+1}-x_{n+1},
\end{aligned}
\right.
\end{equation}
where $\gamma\in (0,\frac{1}{8L})$. The other is called a Backward-reflected-forward-backward splitting algorithm, which combines the Reflected-forward-backward splitting algorithm and the Douglas-Rachford splitting algorithm. The iterative scheme is given by
\begin{equation}\label{algorithm3-three-operator}
\left\{
\begin{aligned}
& x_{n+1}=J_{\gamma A_{1}}z_{n}\\
& y_{n+1}=J_{\gamma A_{2}}(2x_{n+1}-z_{n}-\gamma B(2y_{n}-y_{n-1}))\\
& z_{n+1}=z_{n}+y_{n+1}-x_{n+1},
\end{aligned}
\right.
\end{equation}
where $\gamma\in (0,\frac{1}{22L})$. In summary, we summarize existing algorithms for solving monotone inclusions containing three-operator and beyond in Table \ref{Table-1}.

\begin{table}[h]
\centering
\caption{ Operator splitting algorithms for solving three-operator monotone inclusions and beyond. }
\begin{tabular}{c|c}
\hline
Monotone inclusions & Operator splitting algorithms  \\
\hline
$0\in \sum_{i=1}^{m}A_{i}x+Cx$ (\ref{problem1-m-maximally})& Generalized forward-backward splitting algorithm \cite{Raguet-SIAM-2013}\\
\hline
$0\in \sum_{i=1}^{m}A_{i}x+Bx$ (\ref{problem2-m-maximally})& Relaxed forward-backward splitting algorithm \cite{Banert2012}\\
\hline
$0\in A_{1}x+A_{2}x+Cx$ (\ref{problem1-three-operator}) & Three-operator splitting algorithm \cite{davis2015}\\
\hline
\multirow{4}[1]{*}{$0\in Ax+Bx+Cx$ (\ref{problem2-three-operator})} & Forward-backward-half-forward splitting algorithm \cite{Arias2017A}\\
& Semi-forward-reflected-backward splitting algorithm \cite{Malitsky2020SIAMJO}\\
& Semi-reflected-forward-backward splitting algorithm \cite{Cevher2020SVVA}\\
& Outer reflected forward-backward splitting algorithm \cite{Yu-tang-2020}\\
\hline
\multirow{3}[1]{*}{$0\in A_{1}x+A_{2}x+Bx$ (\ref{problem3-three-operator})} & Forward-reflected-Douglas-Rachford splitting algorithm \cite{Ryu2020JOTA}\\
& Backward-forward-reflected-backward splitting algorithm \cite{Rieger2020AMC}\\
& Backward-reflected-forward-backward splitting algorithm \cite{Rieger2020AMC}\\
\hline
\end{tabular}\label{Table-1}
\end{table}

In this paper, we consider the following four-operator monotone inclusions:
\begin{equation}\label{problem-four-operator}
\textrm{find}\quad x\in H\quad\textrm{such that}\quad 0\in A_{1}x+A_{2}x+Bx+Cx
\end{equation}

Although the four-operator monotone inclusions (\ref{problem-four-operator}) could be viewed as special cases of the three-operator monotone inclusions (\ref{problem2-three-operator}) or (\ref{problem3-three-operator}), there are some drawbacks:

(i) Let $A=A_{1}+A_{2}$, then (\ref{problem-four-operator}) is a special case of (\ref{problem2-three-operator}). Therefore, we can employ the three-operator splitting algorithms \cite{Arias2017A, Malitsky2020SIAMJO, Cevher2020SVVA, Yu-tang-2020} to solve (\ref{problem-four-operator}). However, we have to compute the resolvent $J_{\lambda A}=J_{\lambda(A_{1}+A_{2})}$, $\lambda>0$, which usually doesn't have a closed-form solution.

(ii) Let $\bar{B}=B+C$. Since $C$ is cocoercive, $\bar{B}$ is monotone and $L+\frac{1}{\beta}$-Lipschitz. Then, we can use the three-operator splitting algorithms (\ref{algorithm1-three-operator})-(\ref{algorithm3-three-operator}) to solve (\ref{problem-four-operator}). It is obvious that we do not make full use of the cocoercive property of $C$.

To overcome these drawbacks, in this paper, we introduce and analyze three new splitting algorithms for solving the monotone inclusions problem (\ref{problem-four-operator}). These algorithms are established on the foundation of the Forward-reflected-Douglas-Rachford splitting algorithm \cite{Ryu2020JOTA}, the Backward-forward-reflected-backward splitting algorithm \cite{Rieger2020AMC} and the Backward-reflected-forward-backward splitting algorithm \cite{Rieger2020AMC}. As applications, we study composite monotone inclusions involving a finite sum of maximally monotone operators.

The paper is organized as follows. In Section 2,
we introduce notations and preliminary results in monotone operator theory.
In Section 3, we present three splitting algorithms for solving monotone inclusions involving four operators and prove their convergence.
In Section 4, the proposed algorithms are applied to solve the monotone inclusions problem involving the sum of a finite of maximally monotone operators.  
In Section 5, we conduct some numerical experiments on the Projection problem onto the Minkowski sum of convex sets. Finally, we will give some conclusions.

\section{Preliminaries}

Throughout this paper, let $\mathcal{H}$ be a real Hilbert space, and its scalar product
and the associated norm are denoted by $\langle \cdot , \cdot \rangle$ and $\| \cdot \|$, respectively.
The symbols $\rightharpoonup$ and $\rightarrow$ denote weak and strong convergence, respectively.
$N$ denotes the set of natural numbers. Let $2^{\mathcal{H}}$ be the power set of $\mathcal{H}$.

Let $A:\mathcal{H}\rightarrow 2^{\mathcal{H}}$ be a set-valued operator. The \emph{graph} of $A$ is defined by $gra A=\{(x,u)\in \mathcal{H}\times\mathcal{H}|u\in Ax\}$, the \emph{effective domain} of $A$ is defined by $dom A=\{x\in \mathcal{H}|Ax \neq \emptyset\}$, and the sets of \emph{zeros} of $A$ is defined by $zer A=\{x\in \mathcal{H}| 0\in Ax\}$. The \emph{inverse} of $A$ is $A^{-1}:u\mapsto \{x|u\in Ax\}$. Let $C$ be a nonempty closed convex set of $\mathcal{H}$, the \emph{normal cone} to $C$ is defined by
\begin{equation}\label{normal-cone-C}
   N_{C}:\mathcal{H}\rightarrow2^{\mathcal{H}}:x\mapsto\left\{
\begin{aligned}
&\{u\in \mathcal{H}|(\forall y\in C)\;\langle y-x , u\rangle\leq 0\},\;\textrm{if}\; x\in C;\\
&\emptyset,\; \textrm{ otherwise}.\\
\end{aligned}
\right.
\end{equation}

\begin{definition}(\cite{bauschkebook2017})
(Maximally monotone operator) A set-valued operator $A:\mathcal{H}\rightarrow 2^{\mathcal{H}}$ is \emph{monotone} if
\begin{equation}
(\forall (x,u),(y,v)\in gra A)\quad\langle x-y, u-v \rangle\geq 0,
\end{equation}
and it is said to be \emph{maximally monotone} if there exists no monotone operator $B:\mathcal{H}\rightarrow 2^{\mathcal{H}}$ such that $gra B$ properly contains $gra A$.
\end{definition}

\begin{definition}(\cite{bauschkebook2017})
(Resolvent) Let $A:\mathcal{H}\rightarrow 2^\mathcal{H}$ be a set-valued operator. The \emph{resolvent} $J_{A}:\mathcal{H}\rightarrow 2^\mathcal{H}$ of $A$ is
\begin{equation}
J_{A}=({Id}+A)^{-1},
\end{equation}
where $Id$ denotes the identity operator. When $A$ is maximally monotone, $J_{A}$ is single-valued.
\end{definition}

\begin{definition}(\cite{bauschkebook2017})
(Lipschitz continuous) A single-valued operator $B:\mathcal{H}\rightarrow \mathcal{H}$ is $L$-Lipschitz for some $L >0$ if
\begin{equation}
(\forall x,y\in \mathcal{H})\quad \|Bx\,-\,By\|\leq L \|x\,-\,y\|.
\end{equation}
In particular, it is nonexpansive if the operator $B$ is $1$-Lipschitz.
\end{definition}

\begin{definition}(\cite{bauschkebook2017})
(Cocoercive operator) A single-valued operator $C:\mathcal{H}\rightarrow \mathcal{H}$ is $\beta$-\emph{cocoercive} for some $\beta>0$ if
\begin{equation}
(\forall x,y\in \mathcal{H})\quad\langle x-y, Cx-Cy \rangle\geq \beta\|Cx-Cy\|^2.
\end{equation}
\end{definition}

\begin{definition}(\cite{bauschkebook2017})
(Parallel sum) The \emph{parallel sum} $A_{1}\Box \cdots\Box A_{m}:\mathcal{H}\rightarrow 2^{\mathcal{H}}$ of the operators $A_{i}$, $i=1,\cdots,m$  is
\begin{equation}
A_{1}\Box \cdots\Box A_{m}=(A_{1}^{-1}+\cdots+A_{m}^{-1})^{-1}.
\end{equation}
\end{definition}

We present some useful properties of maximally monotone operators.

\begin{lemma}(\cite{bauschkebook2017})
Let $A:\mathcal{H}\rightarrow 2^{\mathcal{H}}$ be maximally monotone, and let $\gamma >0$. Then $J_{\gamma A}: \mathcal{H}\rightarrow \mathcal{H}$ is firmly nonexpansive, that is,
\begin{equation}
(\forall x,y\in \mathcal{H})\quad \|J_{\gamma A}x-J_{\gamma A}y\|^2 + \|(Id-J_{\gamma A})x-(Id-J_{\gamma A})y \|^2\leq \|x-y\|^2.
\end{equation}
\end{lemma}

\begin{lemma}\label{lemma-maximally}(\cite{bauschkebook2017})
\emph(i) Let $A:\mathcal{H}\rightarrow 2^\mathcal{H}$ be maximally monotone , let $r,y\in \mathcal{H}$, and let $\alpha>0$. Then $A^{-1}$ and $\alpha A(x-r)+y$ are  maximally monotone.

\emph(ii) Let $\mathcal{H}$ and $\mathcal{G}$ be real Hilbert spaces. Let $A:\mathcal{H}\rightarrow 2^\mathcal{H}$ and $B:\mathcal{G}\rightarrow 2^\mathcal{G}$ be maximally monotone. Then $A\times B$ is maximally monotone.
\end{lemma}

\begin{lemma}\label{lemma2-maximally}(\cite{bauschkebook2017})
Let $A:\mathcal{H}\rightarrow 2^\mathcal{H}$ be maximally monotone and $B:\mathcal{H}\rightarrow \mathcal{H}$ be monotone and Lipschitz with $dom B = H$. Then $A+B$ is  maximally monotone.
\end{lemma}

We shall make use of the following Opial's lemma to prove the convergence.

\begin{lemma}\label{opial-lemma}(\cite{bauschkebook2017})
Let $Z$ be a nonempty subset of $\mathcal{H}$ and $\{z_{n}\}$ be a sequence in $\mathcal{H}$. Suppose the following conditions hold:\\
(i) For every $z\in Z$, $\lim_{n\rightarrow\infty}\|z_{n}-z\|$ exists;\\
(ii) Every weak cluster point of $\{z_{n}\}$ belongs to $Z$.\\
Then $\{z_{n}\}$ converges weakly to a point in $Z$.
\end{lemma}

\section{Main algorithms and convergence theorems}

In this section, we present three main algorithms and prove their convergence theorems.

\subsection{Backward-Semi-Forward-Reflected-Backward splitting algorithm}

Now, we are ready to introduce the first splitting algorithm.

\begin{equation}\label{Algorithm 1}
\left\{
\begin{aligned}
& x_{n+1}=J_{\gamma A_{1}}z_{n}\\
& y_{n+1}=J_{\gamma A_{2}}(2x_{n+1}-z_{n}-2\gamma By_{n}+\gamma By_{n-1}-\gamma Cy_{n})\\
& z_{n+1}=z_{n}+y_{n+1}-x_{n+1}
\end{aligned}
\right.
\end{equation}

\begin{remark}
The following iterative algorithms are particular cases of (\ref{Algorithm 1}).

(i) Semi-forward-reflected-backward splitting algorithm \cite{Malitsky2020SIAMJO}: if $A_{1}=0$, (\ref{Algorithm 1}) reduces to
\begin{equation}\label{Algorithm 1-special case-1}
 z_{n+1}=J_{\gamma A_{2}}(z_{n}-2\gamma Bz_{n}+\gamma Bz_{n-1}-\gamma Cz_{n}).
\end{equation}

(ii) Backward-forward-reflected-backward splitting algorithm \cite{Rieger2020AMC}: if $C=0$, (\ref{Algorithm 1}) becomes
\begin{equation}\label{Algorithm 1-special case-2}
\left\{
\begin{aligned}
& x_{n+1}=J_{\gamma A_{1}}z_{n}\\
& y_{n+1}=J_{\gamma A_{2}}(2x_{n+1}-z_{n}-2\gamma By_{n}+\gamma By_{n-1})\\
& z_{n+1}=z_{n}+y_{n+1}-x_{n+1}.
\end{aligned}
\right.
\end{equation}
\end{remark}
Therefore, we call the iterative algorithm (\ref{Algorithm 1}) a backward-semi-forward-reflected-backward splitting algorithm.

\begin{lemma}\label{lemma1}
Suppose that there exist $x,z\in \mathcal{H} $ such that $z-x \in \gamma A_{1}x$ and $x-z \in \gamma(A_{2}+B+C)x$. Let the sequences $\emph\{x_{n}\}$, $\emph\{y_{n}\}$ and $\emph\{z_{n}\}$ be defined in $\emph(\ref{Algorithm 1})$. Then, for all $n\in N$, we have
\begin{equation}\label{the result of lemma1}
\begin{aligned}
\|&z_{n+1}-z\|^2 + 2\gamma\langle By_{n+1}-By_{n},x-y_{n+1} \rangle + \|z_{n+1}-z_{n}\|^2\\
&\leq \|z_{n}-z\|^2 + 2\gamma\langle By_{n}-By_{n-1},x-y_{n} \rangle + 2\gamma\langle By_{n}-By_{n-1},y_{n}-y_{n+1}\rangle + 2\gamma\langle Cy_{n}-Cx,x-y_{n+1}\rangle.
\end{aligned}
\end{equation}
\end{lemma}

\begin{proof}
Since $A_{1}$ is monotone, we have
\begin{equation}\label{the monotonicity of A1}
0\leq \langle (z-x)-(z_{n}-x_{n+1}),x-x_{n+1}\rangle.
\end{equation}
Combining the monotonicity of $A_{2}$ and (\ref{the monotonicity of A1}), it follows that
\begin{equation}\label{the monotonicity of A2}
\begin{aligned}
0\leq \;& \langle (x-z-\gamma (B+C)x)-(2x_{n+1}-z_{n}-y_{n+1}-2\gamma By_{n}+\gamma By_{n-1}-\gamma Cy_{n})),x-y_{n+1}\rangle\\
= \;&\langle (x-z)-(x_{n+1}-z_{n}),x-x_{n+1}\rangle + \langle z_{n+1}-z_{n},z-z_{n+1}\rangle + \gamma\langle By_{n}-Bx,x-y_{n+1}\rangle\\
& +\gamma\langle By_{n}-By_{n-1},x-y_{n+1}\rangle + \gamma\langle Cy_{n}-Cx,x-y_{n+1}\rangle\\
\leq\;& \langle z_{n+1}-z_{n},z-z_{n+1}\rangle + \gamma\langle By_{n}-Bx,x-y_{n+1}\rangle +\gamma\langle By_{n}-By_{n-1},x-y_{n+1}\rangle \\
& + \gamma\langle Cy_{n}-Cx,x-y_{n+1}\rangle.
\end{aligned}
\end{equation}
Using the monotonicity of $B$ yields
\begin{equation}\label{the monotonicity of B}
\begin{aligned}
\gamma\langle By_{n}-Bx,x-y_{n+1}\rangle &= \gamma\langle By_{n}-By_{n+1},x-y_{n+1}\rangle + \gamma\langle By_{n+1}-Bx,x-y_{n+1}\rangle\\
& \leq \gamma\langle By_{n}-By_{n+1},x-y_{n+1}\rangle.
\end{aligned}
\end{equation}
By substituting the estimate (\ref{the monotonicity of B}) into (\ref{the monotonicity of A2}), and using the identity
\begin{equation}
\langle z_{n+1}-z_{n},z-z_{n+1}\rangle = \frac{1}{2}(\|z_{n}-z\|^2-\|z_{n+1}-z_{n}\|^2-\|z_{n+1}-z\|^2),
\end{equation}
the inequality (\ref{the monotonicity of A2}) can be expressed as
\begin{equation}
\begin{aligned}
0\leq \;&\|z_{n}-z\|^2-\|z_{n+1}-z_{n}\|^2-\|z_{n+1}-z\|^2 + 2\gamma\langle By_{n}-By_{n+1},x-y_{n+1}\rangle \\
&+ 2\gamma\langle By_{n}-By_{n-1},x-y_{n+1}\rangle + 2\gamma\langle Cy_{n}-Cx,x-y_{n+1}\rangle.
\end{aligned}
\end{equation}
which implies the claimed inequality (\ref{the result of lemma1}) holds.
\end{proof}

\begin{theorem}\label{Theorem 1}
Let $A_{i}:\mathcal{H}\rightarrow 2^{\mathcal{H}}$, $i=1,2$ be maximally monotone, let $B:\mathcal{H}\rightarrow \mathcal{H}$ be monotone and $L$-Lipschitz continuous, let $C:\mathcal{H}\rightarrow \mathcal{H}$ be $\beta$-cocoercive, and assume that $zer(A_{1}+A_{2}+B+C)\neq \emptyset$. Let $\gamma \in \left(0,\frac{\beta}{2(1+4\beta L)}\right)$. Let $z_{0}$, $y_{0}$, $y_{-1}\in \mathcal{H}$ and consider the sequences $\{x_{n}\}$, $\{y_{n}\}$, $\{z_{n}\}$ defined in \emph{(\ref{Algorithm 1})}. Then, for all $n\in N$, the following hold:\\
\emph{(i)} The sequence $\{z_{n}\}$ converges weakly to a point $\bar{z}\in \mathcal{H}$.\\
\emph{(ii)} The sequences $\{x_{n}\}$, $\{y_{n}\}$ converge weakly to a point $\bar{x}\in \mathcal{H}$.\\
\emph{(iii)} $\bar{x}=J_{\gamma A_{1}}(\bar{z})\in zer(A_{1}+A_{2}+B+C)$.
\end{theorem}

\begin{proof}
We define the set $Z$ in Lemma \ref{opial-lemma} by
\begin{equation}\label{definition of Z}
Z\,:=\{z\in \mathcal{H}\,:\, J_{\gamma A_{1}}z \in zer(A_{1}+A_{2}+B+C),\;(J_{\gamma A_{1}}-\text{Id})z\in \gamma (A_{2}+B+C)(J_{\gamma A_{1}}z) \}.
\end{equation}
Next, we need to prove that $Z$ is nonempty. Note that $zer(A_{1}+A_{2}+B+C)\neq \emptyset$, thus there exist $x \in zer(A_{1}+A_{2}+B+C)$ and $z\in \mathcal{H}$ such that $z-x \in \gamma A_{1}x$ and $x-z \in \gamma(A_{2}+B+C)x$. Using the two inclusions, it gives
\begin{equation}
x=J_{\gamma A_{1}}z \quad \textrm{and} \quad (J_{\gamma A_{1}}-\text{Id})z\in \gamma (A_{2}+B+C)(J_{\gamma A_{1}}z),
\end{equation}
which imply $z\in Z$, that is, the set $Z$ is nonempty.

Let $z\in Z$ and define $x\,:=J_{\gamma A_{1}}z$. By Lemma \ref{lemma1}, for all $n\in N$, we have
\begin{equation}\label{the result of lemma1-}
\begin{aligned}
\|&z_{n+1}-z\|^2 + 2\gamma\langle By_{n+1}-By_{n},x-y_{n+1} \rangle + \|z_{n+1}-z_{n}\|^2\\
&\leq \|z_{n}-z\|^2 + 2\gamma\langle By_{n}-By_{n-1},x-y_{n} \rangle + 2\gamma\langle By_{n}-By_{n-1},y_{n}-y_{n+1}\rangle + 2\gamma\langle Cy_{n}-Cx,x-y_{n+1}\rangle.
\end{aligned}
\end{equation}
By applying Young's inequality and using the nonexpansivity of $\text{Id}-J_{\gamma A_{1}}$, which follows the firm  nonexpansivity of $J_{\gamma A_{1}}$, we obtain
\begin{equation}\label{the equality-1}
\begin{aligned}
\|y_{n}-y_{n-1}\|^2 \;& = \|(z_{n}-z_{n-1}+x_{n})-(z_{n-1}-z_{n-2}+x_{n-1})\|^2\\
& \leq (1+a)\|z_{n}-z_{n-1}\|^2 + (1+\frac{1}{a})\|(x_{n}-z_{n-1})-(x_{n-1}-z_{n-2})\|^2\\
& \leq (1+a)\|z_{n}-z_{n-1}\|^2 + (1+\frac{1}{a})\|z_{n-1}-z_{n-2}\|^2
\end{aligned}
\end{equation}
for some $a>0$. Combining the Lipschitz property of $B$ and (\ref{the equality-1}), the second-last term in (\ref{the result of lemma1-}) can be estimated by
\begin{equation}\label{estimate-1}
\begin{aligned}
& 2\gamma \langle By_{n}-By_{n-1},y_{n}-y_{n+1}\rangle \\
\leq \;& \gamma L(\|y_{n}-y_{n-1}\|^2 + \|y_{n+1}-y_{n}\|^2)\\
\leq \;& (1+\frac{1}{a})\gamma L\|z_{n-1}-z_{n-2}\|^2 + (2+a+\frac{1}{a})\gamma L\|z_{n}-z_{n-1}\|^2 + (1+a)\gamma L\|z_{n+1}-z_{n}\|^2.
\end{aligned}
\end{equation}
Combining the cocoercivity of $C$ and (\ref{the equality-1}), for all $\varepsilon>0$, the last term in (\ref{the result of lemma1-}) can be estimated by
\begin{equation}\label{estimate-2}
\begin{aligned}
& 2\gamma \langle Cy_{n}-Cx,x-y_{n+1}\rangle \\
= \;& 2\gamma \langle Cy_{n}-Cx,x-y_{n} \rangle + 2\varepsilon\langle \frac{\gamma}{\varepsilon}(Cy_{n}-Cx),y_{n}-y_{n+1}\rangle\\
\leq \;& -2\gamma\beta \|Cy_{n}-Cx\|^2 + \frac{\gamma^2}{\varepsilon}\|Cy_{n}-Cx\|^2 + \varepsilon \|y_{n}-y_{n+1}\|^2 - \varepsilon \|\frac{\gamma}{\varepsilon}(Cy_{n}-Cx)- (y_{n}-y_{n+1})\|^2\\
\leq \;& (1+a)\varepsilon \|z_{n+1}-z_{n}\|^2 + (1+\frac{1}{a})\varepsilon \|z_{n}-z_{n-1}\|^2 - \frac{\gamma}{\varepsilon}(2\beta \varepsilon-\gamma)\|Cy_{n}-Cx\|^2.
\end{aligned}
\end{equation}
By substituting the estimates (\ref{estimate-1}), (\ref{estimate-2}) into (\ref{the result of lemma1-}), it follows that
\begin{equation}\label{tranformed result of lemma1}
\begin{aligned}
 \|& z_{n+1}-z\|^2 + 2\gamma\langle By_{n+1}-By_{n},x-y_{n+1} \rangle + \|z_{n+1}-z_{n}\|^2\\
 \leq &   \|z_{n}-z\|^2 + 2\gamma\langle By_{n}-By_{n-1},x-y_{n} \rangle + (1+a)(\gamma L +\varepsilon) \|z_{n+1}-z_{n}\|^2 \\
\; & +  ((2+a+\frac{1}{a})\gamma L + (1+\frac{1}{a})\varepsilon)\|z_{n}-z_{n-1}\|^2 + (1+\frac{1}{a})\gamma L\|z_{n-1}-z_{n-2}\|^2 - \frac{\gamma}{\varepsilon}(2\beta \varepsilon-\gamma)\|Cy_{n}-Cx\|^2.
\end{aligned}
\end{equation}
For convenience, we denote
\begin{equation}
V_{n}\; :=  \|z_{n}-z\|^2 + 2\gamma\langle By_{n}-By_{n-1},x-y_{n} \rangle  +  ((3+a+\frac{2}{a})\gamma L + (1+\frac{1}{a})\varepsilon)\|z_{n}-z_{n-1}\|^2 + (1+\frac{1}{a})\gamma L\|z_{n-1}-z_{n-2}\|^2.
\end{equation}
Hence, (\ref{tranformed result of lemma1}) can be simply expressed as
\begin{equation}\label{energy function}
V_{n+1} + (1-(2+a+\frac{1}{a})(\varepsilon+2\gamma L))\|z_{n+1}-z_{n}\|^2 + \frac{\gamma}{\varepsilon}(2\beta \varepsilon-\gamma)\|Cy_{n}-Cx\|^2 \leq V_{n}.
\end{equation}
In order to ensure that the sequence $\{V_{n}\}$ is nonincreasing, the second and third terms in (\ref{energy function}) shall be positive, from which we derive two upper bounds of $\gamma$: $\gamma< \frac{1-(2+a+\frac{1}{a})\varepsilon}{2L(2+a+\frac{1}{a})}$ and $\gamma<2\beta\varepsilon$. To get the largest interval for $\gamma$, taking these two bounds equal yields $\varepsilon=\frac{1}{(2+a+\frac{1}{a})(1+4\beta L)}>0$. Therefore, in the particular case when $a=1$, we obtain the desired result of $\gamma$.
In addition, (\ref{energy function}) gives
\begin{equation}\label{energy function-}
V_{n+1} + \varepsilon' \|z_{n+1}-z_{n}\|^2 \leq V_{n},
\end{equation}
which implies that
\begin{equation}\label{sum of energy function}
V_{n+1} + \varepsilon' \sum_{i=0}^{n}\|z_{i+1}-z_{i}\|^2 \leq V_{0},
\end{equation}
where $\varepsilon' =1-\frac{1}{1+4\beta L}-8\gamma L$. Next, we need to show that $\{V_{n+1}\}$ has a lower bound. By the nonexpansivity of $J_{\gamma A_{1}}$, the Lipschitz property of $B$ and (\ref{the equality-1}), it follows that
\begin{equation}
\begin{aligned}
& 2\gamma \langle By_{n+1}-By_{n},x-y_{n+1} \rangle\\
\leq &\; \gamma L\|y_{n+1}-y_{n}\|^2 + \gamma L\|(z_{n+1}-x_{n+2})-(z_{n}-x_{n+1})+(x_{n+2}-x)\|^2\\
\leq &\; 2\gamma L\|z_{n}-z_{n-1}\|^2 + 2\gamma L\|z_{n+1}-z_{n}\|^2 + 2\gamma L\|(z_{n+1}-x_{n+2})-(z_{n}-x_{n+1})\|^2 + 2\gamma L\|x_{n+2}-x\|^2\\
\leq &\; 2\gamma L\|z_{n}-z_{n-1}\|^2 + 4\gamma L\|z_{n+1}-z_{n}\|^2 + 2\gamma L\|z_{n+1}-z\|^2,
\end{aligned}
\end{equation}
which gives
\begin{equation}
\begin{aligned}\label{lower bound of Vn}
V_{n+1} & = \|z_{n+1}-z\|^2 + 2\gamma\langle By_{n+1}-By_{n},x-y_{n+1} \rangle  +  (6\gamma L + 2\varepsilon)\|z_{n+1}-z_{n}\|^2 + 2\gamma L\|z_{n}-z_{n-1}\|^2\\
& \geq (1-2\gamma L)\|z_{n+1}-z\|^2 + (2\gamma L + 2\varepsilon)\|z_{n+1}-z_{n}\|^2\\
& \geq (6\gamma L + 2\varepsilon)\|z_{n+1}-z\|^2 \geq 0.
\end{aligned}
\end{equation}
Consequently, $\{V_{n}\}$ converges, together with (\ref{sum of energy function}), (\ref{lower bound of Vn}), which implies $\|z_{n+1}-z_{n}\| \rightarrow 0$ and $\{z_{n}\}$ is bounded. Due to the fact that $x_{n+1}=J_{\gamma A_{1}}z_{n}$ and $x=J_{\gamma A_{1}}z$, the nonexpansivity of $J_{\gamma A_{1}}$ yields $\|x_{n+1}-x_{n}\| \rightarrow 0$ and $\{x_{n}\}$ is bounded. From the identity $ y_{n+1}=z_{n+1}-z_{n}+x_{n+1}$, we deduce that $\|y_{n+1}-y_{n}\| \rightarrow 0$ and $\{y_{n}\}$ is bounded. Hence, we obtain that
\begin{equation}
\lim_{n \rightarrow \infty} \|z_{n}-z\|^2 = \lim_{n \rightarrow \infty} V_{n}
\end{equation}
exists. On the other hand, let $z\in \mathcal{H}$ be a weak cluster point of $\{z_{n}\}$. By the boundedness of $\{x_{k}\}$, it follows that there exists $x\in \mathcal{H}$ such that ($x,z$) is a weak cluster point of $(\{x_{n_{k}}\},\{z_{n_{k}}\})_{k\in N}$. Using the definition of the resolvent operator, (\ref{Algorithm 1}) can be expressed as
\begin{equation}\label{inclusion of the sequence}
\begin{aligned}
& \dbinom{z_{n_k}-z_{n_{k+1}}}{z_{n_k}-z_{n_{k+1}}}-\gamma \dbinom{0}{By_{n_k}-By_{n_{k-1}}}-\gamma \dbinom{0}{By_{n_k}-By_{n_{k+1}}}-\gamma \dbinom{0}{Cy_{n_k}-Cy_{n_{k+1}}}\\
& \in \left(\bm{A}+\bm{B}\right)\dbinom{z_{n_k}-x_{n_{k+1}}}{z_{n_{k+1}}-z_{n_k}+x_{n_{k+1}}},
\end{aligned}
\end{equation}
where $\bm{A}=\begin{bmatrix}
(\gamma A_{1})^{-1}     &  0      \\
0      &  \gamma (A_{2}+B+C)
\end{bmatrix}$ and $\bm{B}=\begin{bmatrix}
0      &  \textrm{-Id}      \\
\textrm{Id}      &  0
\end{bmatrix}$. Since $(\gamma A_{1})^{-1}$ and $\gamma (A_{2}+B+C)$ are maximally monotone by Lemma \ref{lemma-maximally} and \ref{lemma2-maximally}, $\bm{A}+\bm{B}$ is also maximally monotone. Thus, its graph is closed in the weak-strong topology on $\mathcal{H}\times \mathcal{H}$. By taking the limit along the subsequence in (\ref{inclusion of the sequence}), it yields
\begin{equation}
\left\{
\begin{aligned}
& 0\in (\gamma A_{1})^{-1}(z-x)-x\\
& 0\in z-x+\gamma (A_{2}+B+C)x
\end{aligned}
\right.
\Longrightarrow
\left\{
\begin{aligned}
& x=J_{\gamma A_{1}}z\\
& (J_{\gamma A_{1}}-\text{Id})z\in \gamma (A_{2}+B+C)(J_{\gamma A_{1}}z).
\end{aligned}
\right.
\end{equation}
Hence, we deduce that $z\in Z$. According to Lemma \ref{opial-lemma}, it follows that $\{z_{n}\}$ converges weakly to $\bar{z}\in Z$. Let ${x}\in \mathcal{H}$ be a weak cluster point of $\{x_{n}\}$. Taking the limit along the subsequence $(\{x_{n}\},\{z_{n}\})$ of the inclusion $z_{n}-x_{n+1}\in \gamma A_{1}x_{n+1}$ and  yield ${x}=J_{\gamma A_{1}}\bar{z}$. Therefore, the sequence $\{x_{n}\}$ has the unique weak cluster point and converges weakly to $\bar{x}=J_{\gamma A_{1}}\bar{z}$. Finally, Combining the identity $ y_{n+1}=z_{n+1}-z_{n}+x_{n+1}$ and the convergence of the sequence $\{x_{n}\}$ and $\{z_{n}\}$, we derive that $\{y_{n}\}$ converges weakly to a point $\bar{x}\in \mathcal{H}$.
\end{proof}

\begin{remark}
It is worth mentioning that the problem (\ref{problem-four-operator})
can also be solved by applying the Backward-forward-reflected-backward splitting algorithm
(\ref{algorithm2-three-operator}) when $B+C$ in (\ref{Algorithm 1}) is considered as monotone
and $L_{1}$-Lipschitz continuous operator with $L_{1}=L + \frac{1}{\beta}$.
Meanwhile, we derive that the stepsize $\gamma$ in the algorithm (\ref{algorithm2-three-operator})
satisfies $\gamma \in \left(0,\frac{\beta}{8(1+\beta L)}\right)$. Due to the fact that $\frac{\beta}{8(1+\beta L)}<\frac{\beta}{2(1+4\beta L)}$,
the stepsize $\gamma$ in the algorithm (\ref{Algorithm 1}) has a slight improvement.
\end{remark}

\subsection{Backward-Semi-Reflected-Forward-Backward splitting algorithm}

We propose the second splitting algorithms as follows.

\begin{equation}\label{Algorithm 2}
\left\{
\begin{aligned}
& x_{n+1}=J_{\gamma A_{1}}z_{n}\\
& y_{n+1}=J_{\gamma A_{2}}(2x_{n+1}-z_{n}-\gamma B(2y_{n}-y_{n-1})-\gamma Cy_{n})\\
& z_{n+1}=z_{n}+y_{n+1}-x_{n+1}
\end{aligned}
\right.
\end{equation}

\begin{remark}
The following iterative algorithms are special cases of (\ref{Algorithm 2}).

(i) Semi-reflected-forward-backward splitting algorithm \cite{Cevher2020SVVA}: if $A_{1}=0$, (\ref{Algorithm 2}) reduces to
\begin{equation}\label{Algorithm 2-special case-1}
 z_{n+1}=J_{\gamma A_{2}}(z_{n}-\gamma B(2z_{n}-z_{n-1})-\gamma Cz_{n}).
\end{equation}

(ii) Backward-reflected-forward-backward splitting algorithm \cite{Rieger2020AMC}: if $C=0$, (\ref{Algorithm 2}) becomes
\begin{equation}\label{Algorithm 2-special case-2}
\left\{
\begin{aligned}
& x_{n+1}=J_{\gamma A_{1}}z_{n}\\
& y_{n+1}=J_{\gamma A_{2}}(2x_{n+1}-z_{n}-\gamma B(2y_{n}-y_{n-1}))\\
& z_{n+1}=z_{n}+y_{n+1}-x_{n+1}.
\end{aligned}
\right.
\end{equation}
\end{remark}

Therefore, we call the iterative algorithm (\ref{Algorithm 2}) a backward-semi-reflected-forward-backward splitting algorithm. For convenience, we first give the following notations.
\begin{equation}
\hat{x}_{n}=2x_{n}-x_{n-1}, \hat{y}_{n}=2y_{n}-y_{n-1}, \hat{z}_{n}=2z_{n}-z_{n-1}
\end{equation}

\begin{lemma}\label{lemma2}
Suppose that there exist $x,z\in \mathcal{H} $ such that $z-x \in \gamma A_{1}x$ and $x-z \in \gamma(A_{2}+B+C)x$. Let the sequences $\emph\{x_{n}\}$, $\emph\{y_{n}\}$ and $\emph\{z_{n}\}$ be defined in $\emph(\ref{Algorithm 2})$. Then, for all $n\in N$, we have
\begin{equation}\label{the result of lemma2}
\begin{aligned}
\|z_{n+1}- & z  \|^2  +  2\gamma\langle B\hat y_{n}-Bx,y_{n+1}-y_{n} \rangle + 2\|z_{n+1}-z_{n}\|^2 + \|z_{n+1}-\hat z_{n}\|^2\\
\leq \|z_{n} & -z \|^2  + 2\gamma\langle B\hat y_{n-1}-Bx,y_{n}-y_{n-1} \rangle + \|z_{n}-z_{n-1}\|^2 + 2\gamma\langle B\hat y_{n}-B\hat y_{n-1},\hat y_{n}-y_{n+1}\rangle \\
& + 2\gamma\langle Cy_{n-1}-Cy_{n},y_{n+1}-y_{n}\rangle + 2\gamma\langle Cy_{n}-Cx,x-y_{n+1}\rangle.
\end{aligned}
\end{equation}
\end{lemma}

\begin{proof}
Since $A_{1}$ is monotone, we have
\begin{equation}\label{the monotonicity of A1-}
0\leq \langle (z-x)-(z_{n}-x_{n+1}),x-x_{n+1}\rangle.
\end{equation}
Combining the monotonicity of $A_{2}$ and (\ref{the monotonicity of A1}), it follows that
\begin{equation}\label{the monotonicity of A2-}
\begin{aligned}
0\leq \;& \langle (x-z-\gamma (B+C)x)-(2x_{n+1}-z_{n}-y_{n+1}-\gamma B\hat y_{n}-\gamma Cy_{n})),x-y_{n+1}\rangle\\
= \;&\langle (x-z)-(x_{n+1}-z_{n}),x-x_{n+1}\rangle + \langle z_{n+1}-z_{n},z-z_{n+1}\rangle + \gamma\langle B\hat y_{n}-Bx,x-y_{n+1}\rangle\\
& + \gamma\langle Cy_{n}-Cx,x-y_{n+1}\rangle\\
\leq\;& \langle z_{n+1}-z_{n},z-z_{n+1}\rangle + \gamma\langle B\hat y_{n}-Bx,x-y_{n+1}\rangle + \gamma\langle Cy_{n}-Cx,x-y_{n+1}\rangle.
\end{aligned}
\end{equation}
Using the monotonicity of $B$ yields
\begin{equation}\label{the monotonicity of B-}
\begin{aligned}
\gamma\langle B\hat y_{n}-Bx,x-y_{n+1}\rangle &= \gamma\langle B\hat y_{n}-Bx,x-\hat y_{n}\rangle + \gamma\langle B\hat y_{n}-Bx,\hat y_{n}-y_{n+1}\rangle\\
& \leq \gamma\langle B\hat y_{n}-Bx,\hat y_{n}-y_{n+1}\rangle.
\end{aligned}
\end{equation}
By substituting the estimate (\ref{the monotonicity of B-}) into (\ref{the monotonicity of A2-}) and using the identity (\ref{the equality-1}), the inequality (\ref{the monotonicity of A2-}) can be expressed as
\begin{equation}\label{important inequality-1}
\begin{aligned}
0\leq \;&\|z_{n}-z\|^2-\|z_{n+1}-z_{n}\|^2-\|z_{n+1}-z\|^2 + 2\gamma\langle B\hat y_{n}-B\hat y_{n+1},\hat y_{n}-y_{n+1}\rangle \\
&+ 2\gamma\langle B\hat y_{n-1}-Bx,\hat y_{n}-y_{n+1}\rangle + 2\gamma\langle Cy_{n}-Cx,x-y_{n+1}\rangle.
\end{aligned}
\end{equation}
Next, we need to estimate the second-last term in (\ref{important inequality-1}). Since $A_{1}$ is monotone, we obtain
\begin{equation}\label{the monotonicity of A1--}
0\leq \langle (z_{n}-x_{n+1})-(z_{n-1}-x_{n}),x_{n+1}-x_{n}\rangle.
\end{equation}
Combining the monotonicity of $A_{2}$ and (\ref{the monotonicity of A1--}), it follows that
\begin{equation}\label{the monotonicity of A2--}
\begin{aligned}
0\leq \;& \langle (2x_{n+1}-z_{n}-y_{n+1}-\gamma B\hat y_{n}-\gamma Cy_{n})-(2x_{n}-z_{n-1}-y_{n}-\gamma B\hat y_{n-1}-\gamma Cy_{n-1}),y_{n+1}-y_{n}\rangle\\
= \;&\langle (x_{n+1}-z_{n+1}-\gamma B\hat y_{n}-\gamma Cy_{n})-(x_{n}-z_{n}-\gamma B\hat y_{n-1}-\gamma Cy_{n-1}),y_{n+1}-y_{n}\rangle \\
=\;& \langle (x_{n+1}-z_{n})-(x_{n}-z_{n-1}),x_{n+1}-x_{n}\rangle + \langle z_{n}-z_{n+1},z_{n+1}-\hat z_{n}\rangle + \gamma\langle Bx-B\hat y_{n},y_{n+1}-y_{n}\rangle \\
& + \gamma\langle B\hat y_{n-1}-Bx,y_{n}-y_{n-1}\rangle + \gamma\langle B\hat y_{n-1}-Bx,y_{n+1}-\hat y_{n}\rangle  + \gamma\langle Cy_{n-1}-Cy_{n},y_{n+1}-y_{n}\rangle\\
\leq\;&  \langle z_{n}-z_{n+1},z_{n+1}-\hat z_{n}\rangle + \gamma\langle Bx-B\hat y_{n},y_{n+1}-y_{n}\rangle + \gamma\langle B\hat y_{n-1}-Bx,y_{n}-y_{n-1}\rangle\\
& + \gamma\langle B\hat y_{n-1}-Bx,y_{n+1}-\hat y_{n}\rangle  + \gamma\langle Cy_{n-1}-Cy_{n},y_{n+1}-y_{n}\rangle.
\end{aligned}
\end{equation}
By using the  identity
\begin{equation}
\langle z_{n}-z_{n+1},z_{n+1}-\hat z_{n}\rangle = \frac{1}{2}(\|z_{n}-z_{n-1}\|^2 - \|z_{n+1}-z_{n}\|^2 - \|z_{n+1}-\hat z_{n}\|^2 ),
\end{equation}
and reorganizing, we get
\begin{equation}\label{estimate-3}
\begin{aligned}
2\gamma\langle B\hat y_{n-1}-Bx,\hat y_{n}-y_{n+1}\rangle \leq\;& \|z_{n}-z_{n-1}\|^2 - \|z_{n+1}-z_{n}\|^2 - \|z_{n+1}-\hat z_{n}\|^2  + 2\gamma\langle Bx-B\hat y_{n},y_{n+1}-y_{n}\rangle \\
& + 2\gamma\langle B\hat y_{n-1}-Bx,y_{n}-y_{n-1}\rangle   + 2\gamma\langle Cy_{n-1}-Cy_{n},y_{n+1}-y_{n}\rangle.
\end{aligned}
\end{equation}
Substituting the estimate (\ref{estimate-3}) into (\ref{important inequality-1}), the claimed inequality (\ref{the result of lemma2}) holds.
\end{proof}

\begin{theorem}\label{Theorem 2}
Let $A_{i}:\mathcal{H}\rightarrow 2^{\mathcal{H}}$, $i=1,2$ be maximally monotone, let $B:\mathcal{H}\rightarrow \mathcal{H}$ be monotone and $L$-Lipschitz continuous, let $C:\mathcal{H}\rightarrow \mathcal{H}$ be $\beta$-cocoercive, and assume that $zer(A_{1}+A_{2}+B+C)\neq \emptyset$. Let $\gamma \in \left(0,\frac{\beta}{5+(10+\frac{6}{a})\beta L}\right)$, where $a=\frac{17\beta L+10+\sqrt{(17\beta L+10)^2+144\beta^2 L^2}}{6\beta L}$. Let $z_{0}$, $y_{0}$, $y_{-1}\in \mathcal{H}$ and consider the sequences $\{x_{n}\}$, $\{y_{n}\}$, $\{z_{n}\}$ defined in \emph{(\ref{Algorithm 2})}. Then, for all $n\in N$, the following hold:\\
\emph{(i)} The sequence $\{z_{n}\}$ converges weakly to a point $\bar{z}\in \mathcal{H}$.\\
\emph{(ii)} The sequences $\{x_{n}\}$, $\{y_{n}\}$ converge weakly to a point $\bar{x}\in \mathcal{H}$.\\
\emph{(iii)} $\bar{x}=J_{\gamma A_{1}}(\bar{z})\in zer(A_{1}+A_{2}+B+C)$.
\end{theorem}

\begin{proof}
Consider the nonempty set $Z$ defined as (\ref{definition of Z}). Let $z\in Z$ and define $x\,:=J_{\gamma A_{1}}z$. By Lemma \ref{lemma2}, for all $n\in N$, we have
\begin{equation}\label{the result of lemma2-}
\begin{aligned}
\|z_{n+1}- & z  \|^2  +  2\gamma\langle B\hat y_{n}-Bx,y_{n+1}-y_{n} \rangle + 2\|z_{n+1}-z_{n}\|^2 + \|z_{n+1}-\hat z_{n}\|^2\\
\leq \|z_{n} & -z \|^2  + 2\gamma\langle B\hat y_{n-1}-Bx,y_{n}-y_{n-1} \rangle + \|z_{n}-z_{n-1}\|^2 + 2\gamma\langle B\hat y_{n}-B\hat y_{n-1},\hat y_{n}-y_{n+1}\rangle \\
& + 2\gamma\langle Cy_{n-1}-Cy_{n},y_{n+1}-y_{n}\rangle + 2\gamma\langle Cy_{n}-Cx,x-y_{n+1}\rangle.
\end{aligned}
\end{equation}
From the firm nonexpansivity of $J_{\gamma A_{1}}$, it follows that
\begin{equation}\label{the inequality-2}
\begin{aligned}
\|(\hat z_{n-1}-\hat x_{n})-(z_{n}-x_{n+1})\|^2 & \leq 2\|(z_{n-1}-x_{n})-(z_{n}-x_{n+1})\|^2 + 2\|(z_{n-1}-x_{n})-(z_{n-2}-x_{n-1})\|^2\\
& \leq 2\|z_{n}-z_{n-1}\|^2 + 2\|z_{n-1}-z_{n-2}\|^2,
\end{aligned}
\end{equation}
and
\begin{equation}\label{the inequality-3}
\begin{aligned}
\|y_{n}-y_{n-1}\|^2 \;& = \|(z_{n}-z_{n-1}+x_{n})-(z_{n-1}-z_{n-2}+x_{n-1})\|^2\\
& \leq 2\|z_{n}-z_{n-1}\|^2 + 2\|(x_{n}-z_{n-1})-(x_{n-1}-z_{n-2})\|^2\\
& \leq 2\|z_{n}-z_{n-1}\|^2 + 2\|z_{n-1}-z_{n-2}\|^2.
\end{aligned}
\end{equation}
By applying Young's inequality and (\ref{the inequality-2}), we obtain
\begin{equation}\label{the inequality-4}
\begin{aligned}
\|\hat y_{n}-y_{n+1}\|^2 &= \|(\hat z_{n}-\hat z_{n-1}+\hat x_{n})-(z_{n+1}-z_{n}+x_{n+1})\|^2\\
& \leq (1+a)\|z_{n+1}-\hat z_{n}\|^2 + (1+\frac{1}{a})\|(\hat z_{n-1}-\hat x_{n})-(z_{n}-x_{n+1})\|^2\\
& \leq (1+a)\|z_{n+1}-\hat z_{n}\|^2 + 2(1+\frac{1}{a})(\|z_{n}-z_{n-1}\|^2 + \|z_{n-1}-z_{n-2}\|^2)\\
\end{aligned}
\end{equation}
for some $a>0$. Combining (\ref{the inequality-3}) and (\ref{the inequality-4}), we derive that
\begin{equation}\label{the inequality-5}
\begin{aligned}
\|\hat y_{n}-\hat y_{n-1}\|^2 & \leq 2\|y_{n}-y_{n-1}\|^2 + 2\|y_{n}-\hat y_{n-1}\|^2\\
& \leq 4\|z_{n}-z_{n-1}\|^2 + 4(2+\frac{1}{a})\|z_{n-1}-z_{n-2}\|^2 + 4(1+\frac{1}{a})\|z_{n-2}-z_{n-3}\|^3 + 2(1+a)\|z_{n}-\hat z_{n-1}\|^2.
\end{aligned}
\end{equation}
By using the inequality (\ref{the inequality-4}) and (\ref{the inequality-5}), it gives
\begin{equation}\label{estimate-4}
\begin{aligned}
& 2\gamma\langle B\hat y_{n}-B\hat y_{n-1},\hat y_{n}-y_{n+1}\rangle \\
\leq &  \gamma L\|\hat y_{n}-\hat y_{n-1}\|^2 + \gamma L\|\hat y_{n}-y_{n+1}\|^2\\
\leq & 2(3+\frac{1}{a})\gamma L\|z_{n}-z_{n-1}\|^2 + 2(5+\frac{3}{a})\gamma L\|z_{n-1}-z_{n-2}\|^2 + 4(1+\frac{1}{a})\gamma L\|z_{n-2}-z_{n-3}\|^3 \\
&\quad + (1+a)\gamma L\|z_{n+1}-\hat z_{n}\|^2+ 2(1+a)\gamma L\|z_{n}-\hat z_{n-1}\|^2.
\end{aligned}
\end{equation}
Note that $C$ is $\beta$-cocoercive, thus $C$ is $\frac{1}{\beta}$-Lipschitz. By the Lipschitz property of $C$ and (\ref{the inequality-3}), the second-last term in (\ref{the result of lemma2-}) can be estimated by
\begin{equation}\label{estimate-5}
\begin{aligned}
2\gamma\langle Cy_{n-1}-Cy_{n},y_{n+1}-y_{n}\rangle & \leq \frac{\gamma}{\beta}(\|y_{n+1}-y_{n}\|^2 + \|y_{n}-y_{n-1}\|^2)\\
& \leq \frac{\gamma}{\beta}(2\|z_{n+1}-z_{n}\|^2 + 4\|z_{n}-z_{n-1}\|^2 + 2\|z_{n-1}-z_{n-2}\|^2).
\end{aligned}
\end{equation}
Using the cocoercivity of $C$ and (\ref{the inequality-3}), the last term in (\ref{the result of lemma2-}) can be estimated by
\begin{equation}\label{estimate-6}
\begin{aligned}
2\gamma\langle Cy_{n}-Cx,x-y_{n+1}\rangle = &\; 2\gamma\langle Cy_{n}-Cx,x-y_{n}\rangle + 2\varepsilon\langle \frac{\gamma}{\varepsilon}(Cy_{n}-Cx),y_{n}-y_{n+1}\rangle  \\
\leq & -2\gamma\beta\|Cy_{n}-Cx\|^2  + \frac{\gamma^2}{\varepsilon}\|Cy_{n}-Cx\|^2 + \varepsilon\|y_{n+1}-y_{n}\|^2\\
& \quad - \varepsilon\|\frac{\gamma}{\varepsilon}(Cy_{n}-Cx)-(y_{n}-y_{n+1})\|^2 \\
 \leq &\; \varepsilon\|y_{n+1}-y_{n}\|^2 - \frac{\gamma}{\varepsilon}(2\beta\varepsilon-\gamma)\|Cy_{n}-Cx\|^2 \\
 \leq &\; 2\varepsilon\|z_{n+1}-z_{n}\|^2 + 2\varepsilon\|z_{n}-z_{n-1}\|^2 - \frac{\gamma}{\varepsilon}(2\beta\varepsilon-\gamma)\|Cy_{n}-Cx\|^2.
\end{aligned}
\end{equation}
Substituting the estimates (\ref{estimate-4}), (\ref{estimate-5}) and (\ref{estimate-6}) into (\ref{the result of lemma2-}), it follows that
\begin{equation}\label{tranformed result of lemma2}
\begin{aligned}
& \|z_{n+1} - z  \|^2  +  2\gamma\langle B\hat y_{n}-Bx,y_{n+1}-y_{n} \rangle + 2\|z_{n+1}-z_{n}\|^2 + \|z_{n+1}-\hat z_{n}\|^2\\
\leq & \|z_{n}  -z \|^2  + 2\gamma\langle B\hat y_{n-1}-Bx,y_{n}-y_{n-1} \rangle + (2\varepsilon+\frac{2\gamma}{\beta})\|z_{n+1}-z_{n}\|^2 + (1+2\varepsilon+\frac{4\gamma}{\beta}+(6+\frac{2}{a})\gamma L)\|z_{n}-z_{n-1}\|^2 \\
&\quad + (\frac{2\gamma}{\beta}+(10+\frac{6}{a})\gamma L)\|z_{n-1}-z_{n-2}\|^2 + 4(1+\frac{1}{a})\gamma L\|z_{n-2}-z_{n-3}\|^3 + (1+a)\gamma L\|z_{n+1}-\hat z_{n}\|^2\\
&\qquad + 2(1+a)\gamma L\|z_{n}-\hat z_{n-1}\|^2 - \frac{\gamma}{\varepsilon}(2\beta\varepsilon-\gamma)\|Cy_{n}-Cx\|^2.
\end{aligned}
\end{equation}
For convenience, we denote
\begin{equation}
\begin{aligned}
V_{n}\; := &  \|z_{n}  -z \|^2  + 2\gamma\langle B\hat y_{n-1}-Bx,y_{n}-y_{n-1} \rangle + (1+2\varepsilon+\frac{6\gamma}{\beta}+(20+\frac{12}{a})\gamma L)\|z_{n}-z_{n-1}\|^2 \\
&\quad + (\frac{2\gamma}{\beta}+(14+\frac{10}{a})\gamma L)\|z_{n-1}-z_{n-2}\|^2 + 4(1+\frac{1}{a})\gamma L\|z_{n-2}-z_{n-3}\|^2 + 2(1+a)\gamma L\|z_{n}-\hat z_{n-1}\|^2.
\end{aligned}
\end{equation}
Hence, (\ref{tranformed result of lemma2}) can be simply expressed as
\begin{equation}\label{energy function-2}
V_{n+1} + (1-4\varepsilon-\frac{8\gamma}{\beta}-(20+\frac{12}{a})\gamma L)\|z_{n+1}-z_{n}\|^2 + (1-3(1+a)\gamma L)\|z_{n}-\hat z_{n-1}\|^2 + \frac{\gamma}{\varepsilon}(2\beta\varepsilon-\gamma)\|Cy_{n}-Cx\|^2 \leq V_{n}.
\end{equation}
In order to ensure that the sequence $\{V_{n}\}$ is nonincreasing, the second, third and fourth terms in (\ref{energy function-2}) shall be positive, from which we derive three upper bounds of $\gamma$: $\gamma< \frac{1-4\varepsilon}{(20+\frac{12}{a})L+\frac{8}{\beta}}$, $\gamma< \frac{1}{3(1+a)L}$ and $\gamma<2\beta\varepsilon$. To get the largest interval for $\gamma$, first taking the first and third bound equal yields $\varepsilon=\frac{1}{5+(10+\frac{6}{a})\beta L}>0$ and then setting the last two bounds equal yields $a=\frac{17\beta L+10+\sqrt{(17\beta L+10)^2+144\beta^2 L^2}}{6\beta L}>0$. Therefore, we obtain the desired result of $\gamma$.
In addition, (\ref{energy function-2}) gives
\begin{equation}\label{energy function-2-}
V_{n+1} + \varepsilon' \|z_{n+1}-z_{n}\|^2 \leq V_{n},
\end{equation}
which implies that
\begin{equation}\label{sum of energy function-2}
V_{n+1} + \varepsilon' \sum_{i=0}^{n}\|z_{i+1}-z_{i}\|^2 \leq V_{0},
\end{equation}
where $\varepsilon' =1-4\varepsilon-\frac{8\gamma}{\beta}-(20+\frac{12}{a})\gamma L$ and the value of $a$, $\varepsilon$ as the above. Next, we need to estimate the lower bound of the sequence $\{V_{n+1}\}$. From the firm nonexpansivity of $J_{\gamma A_{1}}$, it follows  that
\begin{equation}
\begin{aligned}
\|\bar y_{n}-x\|^2 &= \|(\hat z_{n}-\hat z_{n-1}+\hat x_{n})-(z-z+x)\|^2 \\
& \leq 2\|\hat z_{n}-z\|^2 + 2\|(\hat x_{n}-\hat z_{n-1})-(x-z)\|^2 \\
& \leq 2\|(\hat z_{n}-z_{n+1})+(z_{n+1}-z)\|^2 + 2\|\hat z_{n-1}-z\|^2 \\
& \leq 4\|z_{n+1}-z\|^2 + 4\|z_{n+1}-\hat z_{n}\|^2 + 4\| z_{n-1}-z\|^2 + 4\|z_{n-1}-z_{n-2}\|^2,
\end{aligned}
\end{equation}
which combined the inequality
\begin{equation}
\begin{aligned}
\|z_{n-1}-z\|^2 & = \|(z_{n+1}-z)-(z_{n+1}-z_{n}+z_{n}-z_{n-1})\|^2 \\
& \leq  2\|z_{n+1}-z\|^2 + 2\|(z_{n+1}-z_{n})+(z_{n}-z_{n-1})\|^2 \\
& \leq  2\|z_{n+1}-z\|^2 + 4\|z_{n+1}-z_{n}\|^2 +4\|z_{n}-z_{n-1}\|^2,
\end{aligned}
\end{equation}
it implies
\begin{equation}\label{the inequality-6}
\begin{aligned}
& \|\hat y_{n}-x\|^2  \leq 12\|z_{n+1}-z\|^2 + 16\|z_{n+1}- z_{n}\|^2 + 16\|z_{n}-z_{n-1}\|^2 + 4\|z_{n-1}-z_{n-2}\|^2 + 4\|z_{n+1}-\hat z_{n}\|^2.
\end{aligned}
\end{equation}
Hence, by using the Lipschitz property of $B$, Young's inequality, (\ref{the inequality-3}) and (\ref{the inequality-6}), we conclude that
\begin{equation}
\begin{aligned}
& 2\gamma\langle B\hat y_{n}-Bx,y_{n+1}-y_{n} \rangle \\
& \leq \frac{1}{2}\gamma L\|\hat y_{n}-x\|^2 + 2\gamma L\|y_{n+1}-y_{n}\|^2 \\
& \leq 6\gamma L\|z_{n+1}-z\|^2 + 12\gamma L\|z_{n+1}- z_{n}\|^2 + 12\gamma L\|z_{n}-z_{n-1}\|^2 + 2\gamma L\|z_{n-1}-z_{n-2}\|^2 + 2\gamma L\|z_{n+1}-\hat z_{n}\|^2.
\end{aligned}
\end{equation}
which yields
\begin{equation}\label{lower bound of Vn-2}
\begin{aligned}
V_{n+1} &\geq  (1-6\gamma L) \|z_{n+1}  -z \|^2  + (1+2\varepsilon+\frac{6\gamma}{\beta}+(8+\frac{12}{a})\gamma L)\|z_{n+1}-z_{n}\|^2 \\
&\quad + (\frac{2\gamma}{\beta}+(2+\frac{10}{a})\gamma L)\|z_{n}-z_{n-1}\|^2 + (2+\frac{4}{a})\gamma L\|z_{n-1}-z_{n-2}\|^3 + 2a\gamma L\|z_{n+1}-\hat z_{n}\|^2 \\
& \geq (4\varepsilon+\frac{8\gamma}{\beta}+(14+\frac{12}{a})\gamma L)\|z_{n+1}  -z \|^2 \geq 0.
\end{aligned}
\end{equation}
Consequently, $\{V_{n}\}$ converges, together with (\ref{sum of energy function-2}), (\ref{lower bound of Vn-2}), which implies $\|z_{n+1}-z_{n}\| \rightarrow 0$ and $\{z_{n}\}$ is bounded. Due to the fact that the first and third lines of (\ref{Algorithm 1}) and (\ref{Algorithm 2}) is identical, we deduce that $\|x_{n+1}-x_{n}\| \rightarrow 0$ and $\{x_{n}\}$ is bounded, $\|y_{n+1}-y_{n}\| \rightarrow 0$ and $\{y_{n}\}$ is bounded. Hence, we obtain that
\begin{equation}
\lim_{n \rightarrow \infty} \|z_{n}-z\|^2 = \lim_{n \rightarrow \infty} V_{n}
\end{equation}
exists. On the other hand, (\ref{Algorithm 2}) can be expressed as
\begin{equation}
\begin{aligned}
& \dbinom{z_{n_k}-z_{n_{k+1}}}{z_{n_k}-z_{n_{k+1}}}-\gamma \dbinom{0}{By_{n_{k+1}}-B\hat y_{n_{k}}}-\gamma \dbinom{0}{Cy_{n_k}-Cy_{n_{k+1}}} \in \left(\bm{A}+\bm{B}\right)\dbinom{z_{n_k}-x_{n_{k+1}}}{z_{n_{k+1}}-z_{n_k}+x_{n_{k+1}}},
\end{aligned}
\end{equation}
where $\bm{A}=\begin{bmatrix}
(\gamma A_{1})^{-1}     &  0      \\
0      &  \gamma (A_{2}+B+C)
\end{bmatrix}$ and $\bm{B}=\begin{bmatrix}
0      &  \textrm{-Id}      \\
\textrm{Id}      &  0
\end{bmatrix}$.
The remainder of the proof is analogous to Theorem \ref{Theorem 1}.
\end{proof}

\subsection{Semi-Forward-Reflected-Douglas-Rachford splitting algorithm}

The third splitting algorithm is presented below.

\begin{equation}\label{Algorithm 3}
\left\{
\begin{aligned}
& x_{n+1}=J_{\gamma A_{2}}(x_{n}-\gamma u_{n}-\gamma(2Bx_{n}-Bx_{n-1})-\gamma Cx_{n})\\
& y_{n+1}=J_{\lambda A_{1}}(2x_{n+1}-x_{n}+\lambda u_{n})\\
& u_{n+1}=u_{n}+\frac{1}{\lambda}(2x_{n+1}-x_{n}-y_{n+1})
\end{aligned}
\right.
\end{equation}

\begin{remark}
The following iterative algorithms are special cases of (\ref{Algorithm 3}).

(i) Semi-forward-reflected-backward splitting algorithm \cite{Malitsky2020SIAMJO}: if $A_{1}=0$, (\ref{Algorithm 3}) reduces to
\begin{equation}\label{Algorithm 3-special case-1}
 z_{n+1}=J_{\gamma A_{2}}(z_{n}-2\gamma Bz_{n}+\gamma Bz_{n-1}-\gamma Cz_{n}).
\end{equation}

(ii) Forward-reflected-Douglas-Rachford splitting algorithm \cite{Ryu2020JOTA}: if $C=0$, (\ref{Algorithm 3}) becomes
\begin{equation}\label{Algorithm 3-special case-2}
\left\{
\begin{aligned}
& x_{n+1}=J_{\gamma A_{2}}(x_{n}-\gamma u_{n}-\gamma(2Bx_{n}-Bx_{n-1}))\\
& y_{n+1}=J_{\lambda A_{1}}(2x_{n+1}-x_{n}+\lambda u_{n})\\
& u_{n+1}=u_{n}+\frac{1}{\lambda}(2x_{n+1}-x_{n}-y_{n+1})
\end{aligned}
\right.
\end{equation}
\end{remark}

Therefore, we call the iterative algorithm (\ref{Algorithm 3}) a semi-forward-reflected-Douglas-Rachford splitting algorithm.

\begin{theorem}\label{Theorem 3}
Let $A_{i}:\mathcal{H}\rightarrow 2^{\mathcal{H}}$, $i=1,2$ be maximally monotone, let $B:\mathcal{H}\rightarrow \mathcal{H}$ be monotone and $L$-Lipschitz continuous, let $C:\mathcal{H}\rightarrow \mathcal{H}$ be $\beta$-cocoercive, and assume that $zer(A_{1}+A_{2}+B+C)\neq \emptyset$. Let $\lambda >0$ and $\gamma \in \left(0,\frac{\lambda\beta}{\beta+\lambda(2\beta L+1)}\right)$. Let $x_{0}$, $x_{-1}$, $u_{0}\in \mathcal{H}$ and consider the sequences $\{x_{n}\}$, $\{u_{n}\}$ defined in \emph{(\ref{Algorithm 3})}. Then, for all $n\in N$, the sequences $\{x_{n}\}$ converges weakly to a point $\bar{x}\in zer(A_{1}+A_{2}+B+C)$.
\end{theorem}
\begin{proof}
Let us first define the Hilbert space $\mathcal{K}=\mathcal{H}\times\mathcal{H}$ with inner product and associated norm
\begin{equation}
\begin{aligned}
\langle (x,u),(y,v)\rangle_{\mathcal{K}} &\;:=\frac{1}{\gamma}\langle x,y\rangle -\langle x,v\rangle -\langle y,u\rangle + \lambda\langle u,v\rangle, \\
\|(x,u)\|^2_{\mathcal{K}} &\;:=\frac{1}{\gamma}\|x\|^2 - 2\langle x,u\rangle + \lambda\|u\|^2.
\end{aligned}
\end{equation}
The inner product and norm are proved to be valid by simple calculation.

Let $\bar x\in zer(A_{1}+A_{2}+B+C)$ and $\bar u\in A_{1}\bar x$. Combining these two inclusions yields $-\bar u\in (A_{2}+B+C)\bar x$. Next, we define
\begin{equation}
\begin{aligned}
\bm {A_{1}}y_{n+1} &:=u_{n}+\frac{1}{\lambda}(2x_{n+1}-x_{n}-y_{n+1})\in A_{1}y_{n+1}, \\
\bm {A_{2}}x_{n+1} &:=\frac{1}{\gamma}(x_{n}-x_{n+1})-u_{n}-2Bx_{n}+Bx_{n-1}-Cx_{n}\in A_{2}x_{n+1}, \\
\bm {A_{1}}\bar x &:=\bar u\in A_{1}\bar x,\; \bm {A_{2}}\bar x := -\bar u -B\bar x\in A_{2}\bar x.
\end{aligned}
\end{equation}
By the monotonicity of $A_{1}$ and $A_{2}$, it follows that
\begin{equation}\label{inequality of the sequence}
\begin{aligned}
& \|(x_{n+1},u_{n+1})-(\bar x,\bar u)\|_{\mathcal{K}}^2 \\
& = \|(x_{n},u_{n})-(\bar x,\bar u)\|_{\mathcal{K}}^2 - \|(x_{n+1},u_{n+1})-(x_{n},u_{n})\|_{\mathcal{K}}^2 - 2\langle\bm{A_{2}}x_{n+1}-\bm{A_{2}}\bar x,x_{n+1}-\bar x\rangle \\
&\quad - 2\langle\bm {A_{1}}y_{n+1}-\bm {A_{1}}\bar x,y_{n+1}-\bar x\rangle - 2\langle B\bar x-Bx_{n},\bar x-x_{n+1}\rangle + 2\langle Bx_{n}-Bx_{n-1},x_{n}-x_{n+1}\rangle \\
& \quad   + 2\langle Bx_{n}-Bx_{n-1},\bar x-x_{n}\rangle - 2\langle Cx_{n}-C\bar x,x_{n+1}-\bar x\rangle\\
& \leq \|(x_{n},u_{n})-(\bar x,\bar u)\|_{\mathcal{K}}^2 - \|(x_{n+1},u_{n+1})-(x_{n},u_{n})\|_{\mathcal{K}}^2 - 2\langle B\bar x-Bx_{n},\bar x-x_{n+1}\rangle  \\
& \quad + 2\langle Bx_{n}-Bx_{n-1},x_{n}-x_{n+1}\rangle + 2\langle Bx_{n}-Bx_{n-1},\bar x-x_{n}\rangle - 2\langle Cx_{n}-C\bar x,x_{n+1}-\bar x\rangle.
\end{aligned}
\end{equation}
Since $B$ is monotone, we have
\begin{equation}\label{estimate-7}
\begin{aligned}
- 2\langle B\bar x-Bx_{n},\bar x-x_{n+1}\rangle &= - 2\langle B\bar x-Bx_{n+1},\bar x-x_{n+1}\rangle - 2\langle Bx_{n+1}-Bx_{n},\bar x-x_{n+1}\rangle \\
&\leq - 2\langle Bx_{n+1}-Bx_{n},\bar x-x_{n+1}\rangle.
\end{aligned}
\end{equation}
Using the Lipschitz continuity of $B$ yields
\begin{equation}\label{estimate-8}
2\langle Bx_{n}-Bx_{n-1},x_{n}-x_{n+1}\rangle \leq L(\|x_{n}-x_{n-1}\|^2 + \|x_{n+1}-x_{n}\|^2).
\end{equation}
In addition, from the cocoercivity of $C$,  for all $\varepsilon>0$, we derive that
\begin{equation}\label{estimate-9}
\begin{aligned}
& -2\langle Cx_{n}-C\bar x,x_{n+1}-\bar x\rangle \\
=& -2\langle Cx_{n}-C\bar x,x_{n}-\bar x\rangle - 2\varepsilon\langle \frac{1}{\varepsilon}(Cx_{n}-C\bar x),x_{n+1}-x_{n}\rangle \\
\leq & -2\beta\|Cx_{n}-C\bar x\|^2 + \frac{1}{\varepsilon}\|Cx_{n}-C\bar x\|^2 + \varepsilon\|x_{n+1}-x_{n}\|^2 - \varepsilon\|\frac{1}{\varepsilon}(Cx_{n}-C\bar x)-(x_{n+1}-x_{n})\|^2 \\
\leq &\; \varepsilon\|x_{n+1}-x_{n}\|^2 - (2\beta-\frac{1}{\varepsilon})\|Cx_{n}-C\bar x\|^2 .
\end{aligned}
\end{equation}
Substituting the estimates (\ref{estimate-7}), (\ref{estimate-8}) and (\ref{estimate-9}) into (\ref{inequality of the sequence}), it follows that
\begin{equation}\label{inequality-2 of the sequence}
\begin{aligned}
& \|(x_{n+1},u_{n+1})-(\bar x,\bar u)\|_{\mathcal{K}}^2 + 2\langle Bx_{n+1}-Bx_{n},\bar x-x_{n+1}\rangle + \frac{1}{2}\|(x_{n+1},u_{n+1})-(x_{n},u_{n})\|_{\mathcal{K}}^2\\
&\quad \leq \|(x_{n},u_{n})-(\bar x,\bar u)\|_{\mathcal{K}}^2 + 2\langle Bx_{n}-Bx_{n-1},\bar x-x_{n}\rangle + \frac{1}{2}\|(x_{n},u_{n})-(x_{n-1},u_{n-1})\|_{\mathcal{K}}^2  \\
&\qquad - \frac{1}{2}(\|(x_{n+1},u_{n+1})-(x_{n},u_{n})\|_{\mathcal{K}}^2+\|(x_{n},u_{n})-(x_{n-1},u_{n-1})\|_{\mathcal{K}}^2 )\\
&\qquad  + (L+ \varepsilon)(\|x_{n+1}-x_{n}\|^2 + \|x_{n}-x_{n-1}\|^2) - (2\beta-\frac{1}{\varepsilon})\|Cx_{n}-C\bar x\|^2.
\end{aligned}
\end{equation}
Note that Young's inequality implies
\begin{equation}\label{important inequality-2}
\begin{aligned}
0 \leq \frac{\lambda\gamma(L+\varepsilon)}{\lambda-\gamma}\left(\frac{1}{\lambda}\|x_{n+1}-x_{n}\|^2 - 2\langle x_{n+1}-x_{n},u_{n+1}-u_{n}\rangle + \lambda\|u_{n+1}-u_{n}\|^2\right),\\
0 \leq \frac{\lambda\gamma(L+\varepsilon)}{\lambda-\gamma}\left(\frac{1}{\lambda}\|x_{n}-x_{n-1}\|^2 - 2\langle x_{n}-x_{n-1},u_{n}-u_{n-1}\rangle + \lambda\|u_{n}-u_{n-1}\|^2\right).
\end{aligned}
\end{equation}
Summing (\ref{inequality-2 of the sequence}) and (\ref{important inequality-2}), and denoting
\begin{equation}
V_{n} :=\|(x_{n},u_{n})-(\bar x,\bar u)\|_{\mathcal{K}}^2 + 2\langle Bx_{n}-Bx_{n-1},\bar x-x_{n}\rangle + \frac{1}{2}\|(x_{n},u_{n})-(x_{n-1},u_{n-1})\|_{\mathcal{K}}^2,
\end{equation}
we obtain
\begin{equation}\label{inequality-3 of the sequence}
\begin{aligned}
& \quad V_{n+1} + \frac{\lambda-\gamma-2\lambda\gamma(L+\varepsilon)}{2(\lambda-\gamma)}(\|(x_{n+1},u_{n+1})-(x_{n},u_{n})\|_{\mathcal{K}}^2 \\
& + \|(x_{n},u_{n})-(x_{n-1},u_{n-1})\|_{\mathcal{K}}^2 ) + (2\beta-\frac{1}{\varepsilon})\|Cx_{n}-C\bar x\|^2 \leq V_{n}.
\end{aligned}
\end{equation}
To ensure that the sequence $\{V_n\}$ is nonincreasing, we have $\gamma<\frac{\lambda}{1+2\lambda(L+\varepsilon)}$ and $\varepsilon<\frac{1}{2\beta}$. Combining these two inequalities yields $\gamma<\frac{\lambda\beta}{\beta+\lambda(2\beta L+1)}$. In addition, (\ref{inequality-3 of the sequence}) gives
\begin{equation}\label{inequality-4 of the sequence}
V_{n+1} + \varepsilon'(\|(x_{n+1},u_{n+1})-(x_{n},u_{n})\|_{\mathcal{K}}^2+\|(x_{n},u_{n})-(x_{n-1},u_{n-1})\|_{\mathcal{K}}^2 ) \leq V_{n},
\end{equation}
which implies that
\begin{equation}\label{sum of energy function-3}
V_{n+1} + \varepsilon'\sum_{i=0}^{n}(\|(x_{i+1},u_{i+1})-(x_{i},u_{i})\|_{\mathcal{K}}^2+\|(x_{i},u_{i})-(x_{i-1},u_{i-1})\|_{\mathcal{K}}^2 ) \leq V_{0},
\end{equation}
where $\varepsilon'=\frac{\beta(\lambda-\gamma)-\lambda\gamma(2\beta L+1)}{2\beta(\lambda-\gamma)}$.

On the other hand, we need to estimate the lower bound of the sequence $\{V_{n+1}\}$. Note that the Lipschitz property of $B$ yields
\begin{equation}\label{the inequality-7}
2\langle Bx_{n}-Bx_{n-1},\bar x-x_{n}\rangle \leq L(\|x_{n}-x_{n-1}\|^2 + \|\bar x-x_{n}\|^2),
\end{equation}
and it follows from Young's inequality that
\begin{equation}\label{Young inequality}
\begin{aligned}
&\qquad \langle x_{n}-\bar x,u_{n}-\bar u\rangle \leq \frac{1}{2\lambda}\|x_{n}-\bar x\|^2 + \frac{\lambda}{2}\|u_{n}-\bar u\|^2,\\
& \langle x_{n}-x_{n-1},u_{n}-u_{n-1}\rangle \leq \frac{1}{2\lambda}\|x_{n}-x_{n-1}\|^2 + \frac{\lambda}{2}\|u_{n}- u_{n-1}\|^2.
\end{aligned}
\end{equation}
By combining (\ref{the inequality-7}), (\ref{Young inequality}) and $\gamma<\frac{\lambda}{2\lambda L+1}$, we deduce that
\begin{equation}\label{lower bound of Vn-3}
\begin{aligned}
V_{n+1} &\geq \frac{1}{2}\|(x_{n+1},u_{n+1})-(\bar x,\bar u)\|_{\mathcal{K}}^2 + \frac{1}{2\gamma}\|x_{n+1}-\bar x\|^2 - \langle x_{n+1}-\bar x,u_{n+1}-\bar u\rangle + \frac{\lambda}{2}\|u_{n+1}-\bar u\|^2 \\
& \quad + \frac{1}{2\gamma}\|x_{n+1}-x_{n}\|^2 - \langle x_{n+1}-x_{n},u_{n+1}-u_{n}\rangle + \frac{\lambda}{2}\|u_{n+1}-u_{n}\|^2 - L(\|x_{n+1}-x_{n}\|^2 + \|\bar x-x_{n+1}\|^2) \\
& \geq \frac{1}{2}\|(x_{n+1},u_{n+1})-(\bar x,\bar u)\|_{\mathcal{K}}^2 + (\frac{1}{2\gamma}-\frac{1}{2\lambda}-L)\|x_{n+1}-\bar x\|^2 + (\frac{1}{2\gamma}-\frac{1}{2\lambda}-L)\|x_{n+1}-x_{n}\|^2 \\
& \geq \frac{1}{2}\|(x_{n+1},u_{n+1})-(\bar x,\bar u)\|_{\mathcal{K}}^2 \geq 0.
\end{aligned}
\end{equation}
Hence, $\{V_{n}\}$ converges, together with (\ref{sum of energy function-3}) and (\ref{lower bound of Vn-3}), which implies $x_{n+1}-x_{n}\rightarrow 0$, $u_{n+1}-u_{n}\rightarrow 0$ and the sequence ($\{x_{n}\},\{u_{n}\}$) is bounded. From the identity $u_{n+1}=u_{n}+\frac{1}{\lambda}(2x_{n+1}-x_{n}-y_{n+1})$, we have $x_{n+1}-y_{n+1}\rightarrow 0$. Consequently, applying the above conclusions and the Lipschitz continuity of $B$, we obtain
\begin{equation}
 \lim_{n \rightarrow \infty} \|(x_{n},u_{n})-(\bar x,\bar u)\|_{\mathcal{K}}^2 = \lim_{n \rightarrow \infty} V_{n}
\end{equation}
exists. On the other hand, let ($\bar x,\bar u$) be a weak cluster point of a subsequence $(\{x_{n_k}\},\{u_{n_k}\})_{k\in N}$ of ($\{x_{n}\},\{u_{n}\}$). By the definition of resolvent operator, (\ref{Algorithm 3}) can be expressed as
\begin{equation}\label{inclusion of the sequence-3}
\begin{aligned}
& \dbinom{y_{n_{k+1}}-x_{n_{k+1}}}{\frac{1}{\lambda}(x_{n_{k+1}}-y_{n_{k+1}})} + (\frac{1}{\lambda}-\frac{1}{\gamma}) \dbinom{0}{x_{n_{k+1}}-x_{n_{k}}} + \dbinom{0}{Bx_{n_{k+1}}-Bx_{n_{k}}} - \dbinom{0}{By_{n_k}-By_{n_{k-1}}} \\
& \quad + \dbinom{0}{Cx_{n_{k+1}}-Cx_{n_{k}}}\in \left(\bm{A}+\bm{B}\right)\dbinom{u_{n_{k+1}}}{x_{n_{k+1}}},
\end{aligned}
\end{equation}
where $\bm{A}=\begin{bmatrix}
 A_{1}^{-1}     &  0      \\
0      &  A_{2}+B+C
\end{bmatrix}$ and $\bm{B}=\begin{bmatrix}
0      &  \textrm{-Id}      \\
\textrm{Id}      &  0
\end{bmatrix}$. Since $A_{1}^{-1}$ and $A_{2}+B+C$ are maximally monotone by Lemma \ref{lemma-maximally} and \ref{lemma2-maximally}, $\bm{A}+\bm{B}$ is also maximally monotone. Thus, its graph is closed in the weak-strong topology on $\mathcal{H}\times \mathcal{H}$. By taking the limit along the subsequence in (\ref{inclusion of the sequence-3}), it yields
\begin{equation}
0\in A_{1}^{-1}\bar u - \bar x, \quad 0\in (A_{2}+B+C)\bar x+\bar u.
\end{equation}
Hence, we deduce that $\bar x\in zer(A_1+A_2+B+C)$. According to Lemma \ref{opial-lemma}, it follows that ($\{x_{n}\},\{u_{n}\}$) converges weakly to ($\bar x,\bar u$), which implies that the conclusion of Theorem \ref{Theorem 3} holds.
\end{proof}

\section{Applications}
In this section, we study the problem of finding a zero of the sum of $m$ maximally monotone operators, a monotone Lipschitz operator, and a cocoercive operator. In detail, we consider the following monotone inclusion problem.

\textbf{Problem 4.1} Let $H$ be a real Hilbert space, let $\mathcal{A}_{i}:H\rightarrow 2^{H}$ be maximally monotone, let $\mathcal{B}:H\rightarrow H$ be monotone and $L$-Lipschitz, and let $\mathcal{C}:H\rightarrow H$ be $\beta$-cocoercive. The problem is to
\begin{equation}\label{problem-m-maximally}
\textrm{find}\quad \pmb{x}\in H \quad\textrm{such that}\quad 0\in \sum_{i=1}^{m}\mathcal{A}_{i}\pmb{x}+\mathcal{B}\pmb{x}+\mathcal{C}\pmb{x},
\end{equation}
under the assumption that the set of solutions is nonempty.

Next, we will present the relationship between (\ref{problem-four-operator}) and (\ref{problem-m-maximally}). Let $({\omega}_{i})_{1\leq i\leq m}$ be real numbers in $(0,1]$ such that $\sum_{i=1}^{m}\omega_{i}=1$. Let $\mathbf{\mathcal{H}}=H^{m}$ be the Hilbert direct sum of $H$, and its scalar product and associated norm are defined as $\langle x|y \rangle =\sum_{i=1}^{m}\omega_{i}\langle \pmb{x}_{i} | \pmb{y}_{i} \rangle$ and $\|x\|=\sqrt{\sum_{i=1}^{m}\omega_{i}\left\|\pmb{x}_{i}\right\|^2}$, respectively, where $x=(\pmb{x}_{i})_{1\leq i\leq m}$ and $y=(\pmb{y}_{i})_{1\leq i\leq m}$ are elements of $\mathcal{H}$.

\begin{lemma}\label{techni-lemma}
Let $H$, $\mathcal{B}$, $\mathcal{C}$, $(\mathcal{A}_{i})_{1\leq i\leq m}$ be as in Problem \emph{4.1}, define
\begin{equation}\label{product-space 2}
\begin{aligned}
&V :=\{x=(\textbf{x}_{i})_{1\leq i \leq m}\in \mathcal{H}:\textbf{x}_{1}=\cdots=\textbf{x}_{m}\},\\
&j :H\rightarrow V \subset \mathcal{H}:\textbf{x}\mapsto (\textbf{x},\cdots,\textbf{x}),\\
&A :\mathcal{H}\rightarrow 2^{\mathcal{H}}:x\mapsto \frac{1}{\omega_{1}}\mathcal{A}_{1}\textbf{x}_{1}\times\cdots\times\frac{1}{\omega_{m}}\mathcal{A}_{m}\textbf{x}_{m},\\
&B :\mathcal{H}\rightarrow\mathcal{H}:x\mapsto (\mathcal{B}\textbf{x}_{1},\cdots,\mathcal{B}\textbf{x}_{m}),\\
&C:\mathcal{H}\rightarrow\mathcal{H}:x\mapsto (\mathcal{C}\textbf{x}_{1},\cdots,\mathcal{C}\textbf{x}_{m}).
\end{aligned}
\end{equation}
Then the following statements hold:

\emph{(i)} For any $x=(\textbf{x}_{i})_{1\leq i \leq m}\in \mathcal{H}$, $P_{V}x=j(\sum_{i=1}^{m}\omega_{i}\textbf{x}_{i})$.\\
\emph{(ii)} $N_{V}(x)=
  \left\{
\begin{aligned}
& V^{\bot}=\{x=(\textbf{x}_{i})_{1\leq i \leq m}\in \mathcal{H}:\begin{matrix}\sum_{i=1}^{m}\end{matrix}\omega_{i}\textbf{x}_{i}=0\}, \quad \textrm{ if}\;\pmb{x}\in V; \\
& \emptyset,\quad otherwise.
\end{aligned}
\right.$\\
\emph{(iii)} $A$ is maximally monotone, and for any $\gamma > 0$, $J_{\gamma A}:{(\textbf{x}_{i})_{1\leq i \leq m}\,\mapsto\, {(J_{\gamma\mathcal{A}_{i}/\omega_{i}}\textbf{x}_{i})}}$.\\
\emph{(iv)} $B$ is monotone and $L$-Lipschitz.\\
\emph{(v)} $C$ is $\beta$-cocoercive.\\
\emph{(vi)} For any $\textbf{x}\in H$, $\textbf{x}$ is a solution to Problem \emph{4.1} if and only if $j(\textbf{x})\in zer(A+B+C+N_{V})$.
\end{lemma}

\begin{proof}
(i) This follows from Proposition 26.4 (iii) of \cite{bauschkebook2017}.\\
(ii) Combining Proposition 26.4 (i) and (ii) of \cite{bauschkebook2017}, we can obtain it.\\
(iii) See Proposition 23.18 of \cite{bauschkebook2017}.\\
(iv)\,\&\,(v) They follow from easy computations by combining (\ref{product-space 2}) and the properties of $\mathcal{B}$ and $\mathcal{C}$. \\
(vi) For all $\pmb{x}\in H$, we have
\begin{align}\label{product-space 1}
0\in \sum_{i=1}^{m}\mathcal{A}_{i}\pmb{x}+\mathcal{B}\pmb{x}+\mathcal{C}\pmb{x}
&\Leftrightarrow(\exists{(\pmb{y}_{i})_{1\leq i \leq m}}\in \times_{i=1}^{m}\mathcal{A}_{i}\pmb{x})\, 0=\sum_{i=1}^{m}\pmb{y}_{i}+\mathcal{B}\pmb{x}+\mathcal{C}\pmb{x}\nonumber\\
&\Leftrightarrow(\exists{(\pmb{y}_{i})_{1\leq i \leq m}}\in \times_{i=1}^{m}\mathcal{A}_{i}\pmb{x})\, 0=\sum_{i=1}^{m}\omega_{i}(-\pmb{y}_{i}/\omega_{i}-\mathcal{B}\pmb{x}-\mathcal{C}\pmb{x})\nonumber\\
&\Leftrightarrow(\exists{(\pmb{y}_{i})_{1\leq i \leq m}}\in \times_{i=1}^{m}\mathcal{A}_{i}\pmb{x})
\,(-\pmb{y}_{1}/\omega_{1},\cdots,-\pmb{y}_{m}/\omega_{m})-j(\mathcal{B}\pmb{x})-j(\mathcal{C}\pmb{x})\in V^{\bot}\nonumber\\
&\Leftrightarrow 0\in A(j(\pmb{x}))+B(j(\pmb{x}))+C(j(\pmb{x}))+N_{V}(j(\pmb{x}))\nonumber\\
&\Leftrightarrow j(\pmb{x})\in zer(A+B+C+N_{V}).
\end{align}
\end{proof}

Let $A_{1}=N_{V}$ and $A_{2}=A$ in (\ref{problem-four-operator}). Therefore, we can apply the proposed splitting algorithms in Section 3 to solve the problem (\ref{problem-m-maximally}).

\begin{theorem}\label{m-theorem-1}
  Consider the Problem \emph{4.1}. Let  $(\textbf{z}_{i,0})_{1\leq i\leq m},(\textbf{y}_{i,0})_{1\leq i\leq m},(\textbf{y}_{i,-1})_{1\leq i\leq m}\in H^m$ and set
\begin{equation}\label{m-maximally-algorithm1}
  \left\lfloor
\begin{aligned}
&\textbf{x}_{n+1}=\sum_{j=1}^{m}\omega_{j}\textbf{z}_{j,n}\\
&For\; i=1,\cdots,m\\
&\left\lfloor\begin{aligned}
&\textbf{y}_{i,n+1}=J_{\frac{\gamma}{\omega_i}\mathcal{A}_i}(2\textbf{x}_{n+1}-\textbf{z}_{i,n}-2\gamma\mathcal{B}\textbf{y}_{i,n}+\gamma\mathcal{B}\textbf{y}_{i,n-1}-\gamma\mathcal{C}\textbf{y}_{i,n})\\
&\textbf{z}_{i,n+1}=\textbf{z}_{i,n}+\textbf{y}_{i,n+1}-\textbf{x}_{n+1}
\end{aligned}
\right.
\end{aligned}
\right.
\end{equation}
where $\gamma \in \left(0,\frac{\beta}{2(1+4\beta L)}\right)$. Then $\{\pmb{x}_{n}\}$ converges weakly to a solution of Problem \emph{4.1}.
\end{theorem}

\begin{proof}
 Let $x_{n}=j(\pmb{x}_{n})$, $y_{n}=(\pmb{y}_{i,n})_{1\leq i \leq m}$ and $z_{n}=(\pmb{z}_{i,n})_{1\leq i \leq m}$. By Lemma \ref{techni-lemma} (i), $P_{V}z_{n}=j(\sum_{i=1}^{m}\omega_{i}\pmb{z}_{i,n})$ . Hence, it follows from  Lemma \ref{techni-lemma} that (\ref{m-maximally-algorithm1}) can be written as
 \begin{equation}\label{Algorithm 1-}
\left\{
\begin{aligned}
& x_{n+1}=P_{V}z_{n}\\
& y_{n+1}=J_{\gamma A}(2x_{n+1}-z_{n}-2\gamma By_{n}+\gamma By_{n-1}-\gamma Cy_{n})\\
& z_{n+1}=z_{n}+y_{n+1}-x_{n+1}
\end{aligned}
\right.
\end{equation}
Therefore, the conclusions of Theorem \ref{m-theorem-1} follows directly from Theorem \ref{Theorem 1}.
\end{proof}

\begin{theorem}\label{m-theorem-2}
  Consider the Problem \emph{4.1}. Let  $(\textbf{z}_{i,0})_{1\leq i\leq m},(\textbf{y}_{i,0})_{1\leq i\leq m},(\textbf{y}_{i,-1})_{1\leq i\leq m}\in H^m$ and set
\begin{equation}\label{m-maximally-algorithm2}
  \left\lfloor
\begin{aligned}
&\textbf{x}_{n+1}=\sum_{j=1}^{m}\omega_{j}\textbf{z}_{j,n}\\
&For\; i=1,\cdots,m\\
&\left\lfloor\begin{aligned}
&\textbf{y}_{i,n+1}=J_{\frac{\gamma}{\omega_i}\mathcal{A}_i}(2\textbf{x}_{n+1}-\textbf{z}_{i,n}-\gamma\mathcal{B}(2
\textbf{y}_{i,n}-\textbf{y}_{i,n-1})-\gamma\mathcal{C}\textbf{y}_{i,n})\\
&\textbf{z}_{i,n+1}=\textbf{z}_{i,n}+\textbf{y}_{i,n+1}-\textbf{x}_{
n+1}
\end{aligned}
\right.
\end{aligned}
\right.
\end{equation}
where $\gamma \in \left(0,\frac{\beta}{5+(10+\frac{6}{a})\beta L}\right)$, where $a=\frac{17\beta L+10+\sqrt{(17\beta L+10)^2+144\beta^2 L^2}}{6\beta L}$. Then $\{\pmb{x}_{n}\}$ converges weakly to a solution of Problem \emph{4.1}.
\end{theorem}
\begin{proof}
 Let $x_{n}=j(\pmb{x}_{n})$, $y_{n}=(\pmb{y}_{i,n})_{1\leq i \leq m}$, and $z_{n}=(\pmb{z}_{i,n})_{1\leq i \leq m}$. By Lemma \ref{techni-lemma} (i), $P_{V}z_{n}=j(\sum_{i=1}^{m}\omega_{i}\pmb{z}_{i,n})$ .Hence, it follows from  Lemma \ref{techni-lemma} that (\ref{m-maximally-algorithm2}) can be written as
 \begin{equation}\label{Algorithm 2-}
\left\{
\begin{aligned}
& x_{n+1}=P_{V}z_{n}\\
& y_{n+1}=J_{\gamma A}(2x_{n+1}-z_{n}-\gamma B(2y_{n}-y_{n-1})-\gamma Cy_{n})\\
& z_{n+1}=z_{n}+y_{n+1}-x_{n+1}
\end{aligned}
\right.
\end{equation}
Therefore, the conclusions of Theorem \ref{m-theorem-2} follows directly from Theorem \ref{Theorem 2}.
\end{proof}

\begin{theorem}\label{m-theorem-3}
  Consider the Problem \emph{4.1}. Let  $\textbf{x}_{0},\textbf{x}_{-1}\in H, (\textbf{u}_{i,0})_{1\leq i\leq m}\in H^m$ and set
\begin{equation}\label{m-maximally-algorithm3}
  \left\lfloor
\begin{aligned}
&\textbf{x}_{n+1}=\sum_{j=1}^{m}\omega_{j}(\textbf{x}_{n}-\gamma \textbf{u}_{j,n}-\gamma(2\mathcal{B}\textbf{x}_{n}-\mathcal{B}\textbf{x}_{n-1})-\gamma\mathcal{C}\textbf{x}_{n})\\
&For\; i=1,\cdots,m\\
&\left\lfloor\begin{aligned}
&\textbf{y}_{i,n+1}=J_{\frac{\lambda}{\omega_i}\mathcal{A}_i}(2\textbf{x}_{n+1}-\textbf{x}_{n}+\lambda \textbf{u}_{i,n})\\
&\textbf{u}_{i,n+1}=\textbf{u}_{i,n}+\frac{1}{\lambda}(2\textbf{x}_{n+1}-\textbf{x}_{n}-\textbf{y}_{i,n+1})
\end{aligned}
\right.
\end{aligned}
\right.
\end{equation}
where $\lambda >0$ and $\gamma \in \left(0,\frac{\lambda\beta}{\beta+\lambda(2\beta L+1)}\right)$. Then $\{\pmb{x}_{n}\}$ converges weakly to a solution of Problem \emph{4.1}.
\end{theorem}
\begin{proof}
 Let $x_{n}=j(\pmb{x}_{n})$, $y_{n}=(\pmb{y}_{i,n})_{1\leq i \leq m}$, and $u_{n}=(\pmb{u}_{i,n})_{1\leq i \leq m}$. By Lemma \ref{techni-lemma} (i), $P_{V}(x_{n}-\gamma u_{n}-\gamma(2Bx_{n}-Bx_{n-1})-\gamma Cx_{n})=j(\sum_{i=1}^{m}\omega_{i}(\pmb{x}_{n}-\gamma \pmb{u}_{i,n}-\gamma(2\mathcal{B}\pmb{x}_{n}-\mathcal{B}\pmb{x}_{n-1})-\gamma\mathcal{C}\pmb{x}_{n}))$. Hence, it follows from  Lemma \ref{techni-lemma} that (\ref{m-maximally-algorithm3}) can be written as
 \begin{equation}\label{Algorithm 3-}
\left\{
\begin{aligned}
& x_{n+1}=P_{V}(x_{n}-\gamma u_{n}-\gamma(2Bx_{n}-Bx_{n-1})-\gamma Cx_{n})\\
& y_{n+1}=J_{\lambda A}(2x_{n+1}-x_{n}+\lambda u_{n})\\
& u_{n+1}=u_{n}+\frac{1}{\lambda}(2x_{n+1}-x_{n}-y_{n+1})
\end{aligned}
\right.
\end{equation}
Therefore, the conclusions of Theorem \ref{m-theorem-3} follows directly from Theorem \ref{Theorem 3}.
\end{proof}

\section{Numerical experiments}
In this section, we perform numerical experiments to illustrate the performance of the proposed algorithms. All experiments are conducted on a Lenovo Laptop with an AMD Ryzen 5 2500U CPU (2.00 GHZ) and 8.00 GB RAM.

Suppose that $M_{1},\cdots,M_{k}$ are nonempty, closed and convex sets, the Minkowski sum is defined as follows:
\begin{equation}
M_{1}+\cdots+M_{k}=\{m_{1}+\cdots+m_{k}|m_{1}\in M_{1},\cdots,m_{k}\in M_{k}\}.
\end{equation}
We can show that the Minkowski sum of convex sets is convex by simple calculation. The problem of finding the projection of a point $f\in \mathcal{H}$ onto the Minkowski sum $M_{1}+\cdots+M_{k}$ is defined by
\begin{equation}\label{minkowski sum projection problem}
\min\{\|f-x\|\;|x\in M_{1}+\cdots+M_{k} \}.
\end{equation}
Let $x$ be a solution of (\ref{minkowski sum projection problem}), according to Lemma 4.3.1 of \cite{Banert2012}, we have
\begin{equation}\label{minkowski sum problem transform}
0\in (x-f)+(\partial \delta_{M_{1}}\Box\cdots\Box\partial \delta_{M_{k}})(x),
\end{equation}
where $\partial \delta_{M_{1}}\Box\cdots\Box\partial \delta_{M_{k}}$ denotes the infimal convolution of the indictor functions $\delta_{M_{1}},\cdots,\delta_{M_{k}}$. Let $y\in (\partial \delta_{M_{1}}\Box\cdots\Box\partial \delta_{M_{k}})(x)$, and define
\begin{equation}
\begin{aligned}
\mathcal{C}: (x,y) & \mapsto (x-f,0), \\
\mathcal{B}: (x,y) & \mapsto (y,-x), \\
\mathcal{A}_{i}: (x,y) & \mapsto (0, (\partial \delta_{M_{i}})^{-1}y ), i = 1, \cdots, k. \\
\end{aligned}
\end{equation}
Hence, $x$ satisfies (\ref{minkowski sum problem transform}) is equivalent to $(x,y)\in zer(\sum_{i=1}^{k}\mathcal{A}_{i} + \mathcal{B} + \mathcal{C})$. In particular, for each $i=1, \cdots, k$, $\mathcal{A}_{i}$ is maximally monotone, $\mathcal{B}$ is 1-Lipschitz continuous, and $\mathcal{C}$ is 1-cocoercive. Therefore, we can employ the proposed algorithms in Section 4 to solve (\ref{minkowski sum problem transform}). Due to $B$ is linear, the proposed algorithms (\ref{m-maximally-algorithm1}) and (\ref{m-maximally-algorithm2}) are equivalent.

In the following, we illustrate the performance of the proposed algorithms (\ref{m-maximally-algorithm1}) and (\ref{m-maximally-algorithm3}).

\textbf{Example 5.1}(\cite{Banert2012}) Let $\mathcal{H}=R^2$ and consider the problem (\ref{minkowski sum projection problem}) with the sets $M_{1}=[-2,2]\times\{0\}$, $M_{2}=\{0\}\times[-1,1]$ and $M_{3}=\{(x,y)\in R^2|\|(x,y)\|\leq 1\}$.

Notice that the projection on the single sets are as follows:
\begin{equation}
P_{M_{1}}(x,y)=
\left\{
\begin{aligned}
&(-2,0) \quad \textrm{if}\;x<-2,\\
&(x,0) \quad \;\textrm{if}\;-2\leq x \leq 2,\\
&(2,0) \quad \;\textrm{if}\;x>2,\\
\end{aligned}
\right.
\end{equation}
\begin{equation}
P_{M_{2}}(x,y)=
\left\{
\begin{aligned}
&(0,-1) \quad \textrm{if}\;y<-1,\\
&(0,y) \quad \;\textrm{if}\;-1\leq y \leq 1,\\
&(0,1) \quad \;\textrm{if}\;y>1,\\
\end{aligned}
\right.
\end{equation}
\begin{equation}
P_{M_{3}}(x,y)=
\left\{
\begin{aligned}
&(x,y) \qquad\qquad\quad \textrm{if}\;\|(x,y)\|\leq 1,\\
&\frac{1}{\|(x,y)\|}(x,y)\quad \textrm{if}\; \|(x,y)\|>1.
\end{aligned}
\right.
\end{equation}

Let $f=(1,-4)$, $f=(2,7)$, and $f=(6,-4)$, then  the projection of $f$ is $\bar{x}=(1,-2)$,
$\bar{x}=(2,2)$, and $\bar{x}=(2.8,-1.6)$, respectively.
We terminate the algorithm when $\|x_{n}-\bar{x}\|\leq \varepsilon$, where $\varepsilon = 10^{-6}$.
We use ``Iter" to denote the iteration numbers and "Time(s)" to denote the elapsed CPU time (in seconds).
The obtained results are reported in Table \ref{Table-2}. It can be seen from Table \ref{Table-2} that the larger the iteration parameter,
the faster the proposed iterative algorithms (\ref{m-maximally-algorithm1}) and (\ref{m-maximally-algorithm3}) converge.
For the iterative algorithm (\ref{m-maximally-algorithm3}), when the parameter $\lambda$ increases, although the range of the parameter $\gamma$ increases,
the number of iteration numbers increases instead. For the choice of $\lambda = 0.5$,
it can be seen from the results that the algorithm (\ref{m-maximally-algorithm3}) is better than the algorithm (\ref{m-maximally-algorithm1}).

\begin{table}[ht]
\footnotesize
\centering
\caption{Numerical results of the proposed algorithms.}
\begin{tabular}{c|c|c|cccccccc}
\hline
\multirow{2}[1]{*}{The proposed algorithms} & \multirow{2}[1]{*}{$\lambda$} & \multirow{2}[1]{*}{$\gamma$}  &  \multicolumn{2}{c}{$f=(6,-4)$} && \multicolumn{2}{c}{$f=(1,-4)$} && \multicolumn{2}{c}{$f=(2,7)$}  \\
\cline{4-11}
 & & & Iter & Time(s) && Iter & Time(s) && Iter & Time(s)\\
\hline
\multirow{5}[1]{*}{(\ref{m-maximally-algorithm1})} & \multirow{5}[1]{*}{-}
&  $0.02$ & $941$  & $0.71$&& $946$ & $0.73$&& $1110$ & $0.72$\\
& & $0.04$ & $564$  & $0.71$&& $566$ & $0.73$&& $558$ & $0.71$ \\
& & $0.06$ & $378$  & $0.72$&& $379$ & $0.87$&& $374$ & $0.72$ \\
& & $0.08$ & $285$  & $0.74$&& $240$ & $0.77$&& $282$ & $0.72$\\
& & $0.1$ & $229$  & $0.72$&& $193$ & $0.74$&& $226$ & $0.73$ \\
\hline
\multirow{16}[1]{*}{(\ref{m-maximally-algorithm3})} & \multirow{4}[1]{*}{$0.5$}  &  $0.05$ & $457$ & $0.90$ && $456$ & $0.93$&& $457$ & $0.88$ \\
& & $0.1$ & $250$  & $ 0.90$&&  $250$ &0.92&&   $250$ & $0.89$\\
& & $0.15$ & $180$   & $0.91$&&   $149$ & $0.91$&& $179$ & $0.90$ \\
& & $0.2$ & $143$   & $0.90$&&   $142$ & $0.90$& &$166$ & $0.88$ \\
\cline{2-11}
& \multirow{6}[1]{*}{$2$}  &  $0.05$ & $718$ & $0.93$ && $889$ & $0.91$&& $1306$ & $0.91$ \\
& & $0.1$ & $501$  & $ 0.90$&&  $446$ &0.90&&   $592$ & $0.90$\\
& & $0.15$ & $317$   & $0.93$&&   $360$ & $0.91$&& $383$ & $0.89$ \\
& & $0.2$ & $189$   & $0.91$&&   $276$ & $0.92$& &$293$ & $0.92$ \\
& & $0.25$ & $226$   & $0.88$&&   $228$ & $0.90$& &$250$ & $0.89$ \\
& & $0.28$ & $208$   & $0.91$&&   $213$ & $0.92$& &$226$ & $0.91$ \\
\cline{2-11}
& \multirow{6}[1]{*}{$5$}  &  $0.05$ & $1759$ & $0.92$ && $1691$ & $0.91$&& $1756$ & $0.91$ \\
& & $0.1$ & $1209$  & $ 0.89$&&  $946$ &0.90&&   $914$ & $0.91$\\
& & $0.15$ & $738$   & $0.89$&&   $806$ & $0.91$&& $797$ & $0.91$ \\
& & $0.2$ & $678$   & $0.89$&&   $679$ & $0.89$& &$670$ & $0.89$ \\
& & $0.25$ & $581$   & $0.92$&&   $547$ & $0.90$& &$621$ & $0.91$ \\
& & $0.31$ & $481$   & $0.89$&&   $510$ & $0.93$& &$531$ & $0.89$ \\
\hline
\end{tabular}\label{Table-2}
\end{table}

\section{Conclusions}
In this paper, we considered the monotone inclusions with a sum of four operators, in which two of them are maximally monotone, one is monotone Lipschitz, and one is cocoercive.
We introduced three new splitting algorithms to solve it and analyzed their convergence.
As applications, we considered composite monotone inclusion problems. To solve this monotone inclusion, we transformed it into the formulation of (\ref{product-space 1}) by the technique of product space.
We evaluated the performance of the proposed algorithms on the Projection on the Minkowski sums of convex sets problem (\ref{minkowski sum projection problem}). Numerical experiments demonstrated the effectiveness and efficiency of the proposed algorithms.

\section*{Funding}

This work was funded by the National Natural Science Foundations of China (12061045, 11661056).


\end{document}